\documentclass[11pt,reqno]{amsart}
\usepackage[a4paper,margin=27mm]{geometry}
\usepackage[T1]{fontenc}
\usepackage{lmodern}
\usepackage{amsmath,amssymb,amsthm,mathtools}
\usepackage{mathrsfs}
\usepackage{microtype}
\usepackage{enumitem}
\usepackage[hidelinks]{hyperref}
\usepackage{bookmark}
\allowdisplaybreaks[1]
\hypersetup{
 pdftitle={Non-vanishing of primitive cubic, quartic and sextic L-functions},
 pdfauthor={Ziwei Hong},
 pdfsubject={Twisted first moments and lower bounds for non-vanishing proportions},
 pdfkeywords={cubic L-functions, quartic L-functions, sextic L-functions, non-Kummer, truncated exponentials, non-vanishing}
}
\numberwithin{equation}{section}
\newtheorem{theorem}{Theorem}[section]
\newtheorem{proposition}[theorem]{Proposition}
\newtheorem{lemma}[theorem]{Lemma}
\newtheorem{corollary}[theorem]{Corollary}
\theoremstyle{definition}

\theoremstyle{remark}
\newtheorem{remark}[theorem]{Remark}
\newcommand{\F}{\mathbb F}
\newcommand{\C}{\mathbb C}
\newcommand{\E}{\mathbb E}
\newcommand{\Aq}{\mathcal A_q}
\newcommand{\AQ}{\mathcal A_{q^2}}
\newcommand{\Cfour}{\mathcal C_4(g)}
\newcommand{\Lq}{L_q}
\newcommand{\Lpol}{\mathcal L_q}

\newcommand{\norm}[1]{\lVert#1\rVert}

\DeclareMathOperator{\cond}{cond}
\DeclareMathOperator{\rad}{rad}
\DeclareMathOperator{\Rea}{Re}
\DeclareMathOperator{\Res}{Res}

\newcommand{\sumchi}{\sum_{\chi\in\Cfour}}
\newcommand{\sumF}{\sum_{\substack{F\in\AQ\text{ square-free}\\(F,F^\sigma)=1}}}

\title[Non-vanishing of cubic, quartic and sextic $L$-functions]{Non-vanishing of primitive cubic, quartic and sextic $L$-functions}
\author{Ziwei Hong}
\address{School of Mathematics, Renmin University of China, Beijing, P.R. China}
\address{School of Mathematics and Statistics, UNSW Sydney, Sydney, Australia}
\email{hongziwei@live.com}
\date{11 October 2026}

\begin{document}
\begin{abstract}
We establish lower bounds for the non-vanishing proportions of primitive
cubic, quartic and sextic Dirichlet $L$-functions over $\F_q[T]$ in the
non-Kummer setting. Adapting truncated-exponential techniques
for shifted moments, we improve the cubic lower bound of David, Florea and
Lal\'in and obtain positive lower bounds for quartic and sextic
families.
For each family, the lower limit of the non-vanishing proportion
exceeds $e^{-e^{53}}$ as the conductor degree tends to infinity with
$q$ fixed. We use a common choice of truncation parameters and
explicitly control the exceptional sets. The cubic argument uses the
published twisted first moment. We also establish uniform twisted first moments in the
quartic and sextic cases, yielding an asymptotic formula for the
sextic first moment.
\end{abstract}
\maketitle

\section{Introduction}\label{sec:introduction}

The central non-vanishing problem asks how often the value
$L(1/2,\chi)$ is nonzero as $\chi$ varies in an arithmetic family.
A quantitative form of this problem is to obtain a positive lower
bound for the proportion of nonzero central values as the conductor
grows. The answer depends on the ground field, the order of the
characters, and the local conditions defining the family; results
for all primitive characters need not give a positive proportion
in a fixed-order subfamily.

For primitive Dirichlet characters modulo a large integer,
Iwaniec and Sarnak~\cite{IS99} established positive-proportion
non-vanishing by mollification. More recently, Durkan and
Pearce-Crump~\cite{DP26} proved unconditionally a proportion of at
least $3/8-o(1)$ for all primitive characters modulo
$N\not\equiv2\pmod4$ as $N\to\infty$, without restricting their order.

For fixed-order families over the rationals, Gao and Zhao~\cite{GZ22}
proved positive-proportion non-vanishing under the generalized
Riemann hypothesis for primitive cubic Dirichlet characters of
conductor coprime to $3$ and primitive quartic Dirichlet characters
of odd conductor whose squares remain primitive. They also proved,
under the same hypothesis, positive-proportion non-vanishing in a
sextic Hecke family over $\mathbb Q(\omega)$, where $\omega$ is a
primitive cube root of unity~\cite{GZsextic24}.

Unconditionally, David, de Faveri, Dunn, and Stucky~\cite{DdFDS26}
proved positive-proportion non-vanishing for cubic Hecke
$L$-functions attached to square-free parameters
$c\in\mathbb Z[\omega]$ with $c\equiv1\pmod9$.
Castillo, de Faveri, and Dunn~\cite{CdFD26} obtained the corresponding
quartic result over $\mathbb Q(i)$ for square-free residue-symbol
parameters $c\in\mathbb Z[i]$ with $c\equiv1\pmod{(1+i)^7}$.
More generally, Diaconu, Ion, Pa\c{s}ol, and Popa~\cite{DIPP}
used multiple Dirichlet series to evaluate twisted moments and
prove positive-proportion non-vanishing in higher-order Hecke
residue-symbol families over number fields containing the
$2r$-th roots of unity, for $r\ge3$.

Over function fields, David, Florea, and Lal\'in~\cite{DFL22}
evaluated cubic first moments in the Kummer and non-Kummer settings.
Their non-vanishing result~\cite{DFL21} for the full non-Kummer
cubic family combines a mollified first moment with an upper bound
of the correct order for the mollified second moment.
Lin~\cite{Lin} computed one-level densities for cubic and quartic
families in the Kummer setting and for cubic, quartic, and sextic
families in the non-Kummer setting. In the Kummer setting,
David, Florea, and Lal\'in~\cite{DFL25} proved positive-proportion
non-vanishing for square-free residue-symbol families of fixed
order $\ell\ge3$, with $q\equiv1\pmod{2\ell}$ and parameter
degrees not divisible by $\ell$, by extending the Fourier support in a one-level density calculation beyond $(-1,1)$.

We work with primitive characters of orders $3$, $4$, and $6$ in
the non-Kummer setting. The base field lacks the relevant roots
of unity, so the Kummer results above do not directly cover our
families; passing to the quadratic constant-field extension changes
both the $L$-functions and the averaging family. We improve the
explicit cubic non-vanishing bound and obtain positive lower bounds
for the quartic and sextic families defined below. For the latter
two families, we also establish uniform twisted first moments.

Throughout, $q$ is a fixed odd prime power. For $\ell\in\{3,4,6\}$,
assume
\[
 q\equiv-1\pmod\ell.
\]
Let $\mathcal C_\ell(g)$ be the family of primitive characters of exact
order $\ell$ whose associated cyclic extensions have genus $g$, with
\begin{equation}\label{intro:conductor}
 \begin{cases}
 \text{no additional condition},&\ell=3,\\
 \cond(\chi^2)=\cond(\chi),&\ell=4,\\
 \cond(\chi^2)=\cond(\chi^3)=\cond(\chi),&\ell=6.
 \end{cases}
\end{equation}
Thus the cubic family is the full primitive cubic family, whereas the
quartic and sextic families require local order $\ell$ at every
ramified prime. These conductor conditions are part of our theorems.

Put $Q=q^2$, and let $\sigma$ act on coefficients by $a\mapsto a^q$.
The parameterization of \cite[Lemma~3.6]{Lin} gives
\begin{equation}\label{intro:family}
 \begin{split}
 \mathcal C_\ell(g)=\biggl\{
 \chi_F=\left.\left(\frac{\cdot}{F}\right)_\ell\right|_{\F_q[T]}:
 &\ F\in\F_Q[T]\text{ monic and square-free},\\
 &\ \deg F=m,\quad(F,F^\sigma)=1\biggr\},
 \qquad m=\frac{g}{\ell-1}+1.
 \end{split}
\end{equation}
The conductor is $FF^\sigma$ and has degree $2m$. Every character is
even, and the completed $L$-polynomial has degree $2m-2$.
In particular, $g=(\ell-1)(m-1)$; the limits below are taken through
these admissible genera. Conjugate characters are counted separately.
We write $|f|=q^{\deg f}$ for $f\in\F_q[T]$. For
$\chi\in\mathcal C_\ell(g)$, denote its Dirichlet $L$-function by
$L_q(s,\chi)$.

\subsection{Main results}

The arithmetic factors in the non-vanishing bounds have a common
form. For a monic irreducible polynomial $P\in\F_q[T]$, put
\begin{equation}\label{eq:rho}
 \rho_P=
 \begin{cases}
 1,&\deg P\text{ odd},\\[1mm]
 \displaystyle\frac{1}{1+2|P|^{-1}},&\deg P\text{ even},
 \end{cases}
 \qquad \rho(h)=\prod_{P\mid h}\rho_P.
\end{equation}
Define $d_v$ and $A_{\ell,q,G}$ by
\begin{align}
 \frac{ e^{-x-\frac{x^2}{2}}}{1-x}
 &= e^{\sum_{v\ge3}\frac{x^v}{v}}
   =\sum_{v\ge0}d_vx^v,\qquad |x|<1,\label{eq:GZlocal}\\
 A_{\ell,q,G}
 &=\prod_P\left(1+\rho_P\sum_{k\ge1}
                    \frac{d_{\ell k}}{|P|^{\ell k/2}}\right).
                    \label{intro:arithmetic}
\end{align}
Here $0\le d_v\le1$, with $d_3=1/3$, $d_4=1/4$, and $d_6=2/9$.
Let $\zeta_q$ denote the zeta function of $\F_q[T]$.
The products converge absolutely and satisfy
\begin{equation}\label{intro:arithmetic-bounds}
 1<A_{\ell,q,G}\le\zeta_q(\ell/2).
\end{equation}

\begin{theorem}\label{thm:main-non-vanishing}
Let $\ell\in\{3,4,6\}$, and fix an odd prime power
$q\equiv-1\pmod\ell$. For the family \eqref{intro:family},
\begin{equation}\label{intro:non-vanishing}
 \liminf_{\substack{g\to\infty\\\ell-1\mid g}}
 \frac{\#\{\chi\in\mathcal C_\ell(g):L_q(1/2,\chi)\ne0\}}
      {|\mathcal C_\ell(g)|}
 \ge\frac{A_{\ell,q,G}}{2 e^{e^{52}+3}}
 >\frac{1}{2 e^{e^{52}+3}}
 > e^{-e^{53}}.
\end{equation}
\end{theorem}

The bounds are small; we do not determine the actual non-vanishing
proportions or assert that these proportions converge.

For the full non-Kummer cubic family, David, Florea, and
Lal\'in~\cite[Remark~1.8]{DFL21} obtained the explicit lower bound
$0.4718 e^{-e^{182}}$. Since
\begin{equation}\label{intro:comparison}
 \frac{1}{2 e^{e^{52}+3}}
 >0.4718 e^{-e^{182}},
\end{equation}
Theorem~\ref{thm:main-non-vanishing} therefore improves their stated
constant for the same family.

The quartic and sextic proofs use the following uniform moment
formula. For $\ell\in\{4,6\}$, let $h_\ell^\ast$ be the unique monic
polynomial with prime exponents in $\{0,\ldots,\ell-1\}$ such that
$hh_\ell^\ast$ is an $\ell$th power. Let $c_q$ be the positive
family-counting constant in \eqref{eq:cq}, and set
\begin{equation}\label{intro:moment-main}
 \mathcal M_{\ell,q}(h)
 =\frac{c_q}{\sqrt{|h_\ell^\ast|}}
    \sum_{A\text{ monic}}\frac{\rho(hA)}{|A|^{\ell/2}}.
\end{equation}

\begin{theorem}[Uniform twisted first moments]\label{thm:main-moments}
For $\ell=4$ or $6$, fix an odd prime power $q\equiv-1\pmod\ell$,
and put $g=(\ell-1)(m-1)$. For every $\varepsilon>0$, uniformly in
monic $h$ with $\deg h\le m/10$,
\begin{align}\label{intro:twisted-moments}
 \sum_{\chi\in\mathcal C_\ell(g)}\chi(h)L_q(1/2,\chi)
 ={}&q^{2m}\mathcal M_{\ell,q}(h)\notag\\
 &+O_{q,\varepsilon}\!\left(
 q^{(\alpha_\ell+\varepsilon)m}
 |h_\ell^\ast|^{\beta_\ell+\varepsilon}|h|^\varepsilon\right),
\end{align}
where
\[
 (\alpha_4,\beta_4)=\left(\frac95,\frac1{10}\right),
 \qquad
 (\alpha_6,\beta_6)=\left(\frac{13}{7},\frac1{14}\right).
\]
\end{theorem}

The more precise statements appear in Theorems~\ref{input:first}
and~\ref{sextic:tfm}. Taking $h=1$ in the sextic formula gives
Corollary~\ref{sextic:first-moment}, an asymptotic for the untwisted
first moment with an error of size $q^{(13/7+\varepsilon)m}$.
The cubic argument uses \cite[Proposition~8.1]{DFL21},
restated in Proposition~\ref{cubic:twisted}. Its error term already
suffices for the mollifier used here. The improvement in the cubic
constant comes from the truncation parameters and the exceptional-set
estimates in Section~\ref{sec:cubic}.

\subsection{Method of proof}

The twisted-moment calculation extends earlier applications of
double Dirichlet series. Over the rationals, Gao and Zhao~\cite{GZ24}
obtained smoothed asymptotics for primitive cubic, quartic, and
sextic families with full local order at every ramified prime under
the generalized Lindel\"of hypothesis. Over function fields,
Hong and Fang~\cite{HF} evaluated a non-Kummer cubic first moment,
and Hong~\cite{Hong} applied this method
to the quartic family in \eqref{intro:conductor}, obtaining an error
of size $q^{(3/5+\varepsilon)g}$. We retain the dependence on a short
twist and extend the construction to sextic characters.

The required Gauss sum generating-series estimates come from
David, Florea, and Lal\'in~\cite{DFL25}, applied over the quadratic
extension of the constant field. Character averages, the functional
equation, and a restricted Gauss sum series give three domains of
holomorphy. Bochner's tube theorem continues the pole-removed series
to their convex hull; uniform interpolation and coefficient
extraction then give Theorem~\ref{thm:main-moments}.

For non-vanishing, we adapt the truncated-exponential framework
of Curran~\cite[Section~1 and Lemma~2.2]{Curran24} for shifted zeta
moments under the Riemann hypothesis, and its Dirichlet analogue
under the generalized Riemann hypothesis due to Gao and
Zhao~\cite[Sections~2.7, 3 and~5.9]{GZ}. The lower-bound principles underlying this
approach were developed by Radziwi\l\l{} and
Soundararajan~\cite{RS13} and Heap and Soundararajan~\cite{HS22}.
The prime-block truncations and control of exceptional sets also
occur in Harper's method for sharp conditional upper
bounds~\cite{Harper13}.

Our mollifier is a product of squared truncated exponentials of
prime-block sums, with a short prime-square correction in the first
block. We partition the family according to the sizes of these sums.
On the good set, the truncated exponentials approximate the
corresponding exponentials and give a uniform pointwise bound for
the mollified central value. On the exceptional sets, we apply
logarithmic bounds from Weil's theorem at a shorter scale. The
prime-square correction is fixed across scales, so it cancels when
the shorter prime sum is paired with the mollifier. High moments
of the first large block control the exceptional contribution;
the initial exceptional set is handled separately by a rough
fourth-moment bound and its small cardinality.

Uniformity in the twist permits summation against the mollifier
coefficients, giving the full mollified first moment. Subtracting
the absolute contribution of the exceptional sets leaves a lower
bound for the modulus of the good-set average. The pointwise bound
on this set then gives Theorem~\ref{thm:main-non-vanishing} by the
triangle inequality. All three orders use the parameters in
\eqref{eq:GZparameters}, for which $e^{-52}<\theta_J\le e^{-51}$
and $L/m<0.005248338$. The existing twisted first-moment estimates
give a power saving for this mollifier length in each case.
The finite-orthogonality and exceptional-set estimates also hold
with these parameters. On the good set we have
$e^{\theta_J^{-1}+3}<e^{e^{52}+3}$, and its first moment retains
at least half of $c_qA_{\ell,q,G}q^{2m}$. This gives the common
denominator in \eqref{intro:non-vanishing}; the arithmetic factor
$A_{\ell,q,G}$ records the dependence on the character order.

Section~\ref{sec:prelim} records the quartic parameterization and
preliminary estimates. Section~\ref{sec:first} proves the quartic
twisted first moment, and Section~\ref{sec:GZ} establishes its explicit
non-vanishing bound. Section~\ref{sec:sextic} treats the sextic case.
Section~\ref{sec:cubic} applies the same truncation parameters
to the published cubic inputs and proves the improvement
in \eqref{intro:comparison}.

\section{Preliminaries}\label{sec:prelim}

Throughout Sections~\ref{sec:prelim}--\ref{sec:GZ}, $q$ is a fixed odd
prime power with $q\equiv3\pmod4$.

\subsection{Notation and the family}

Write $Q=q^2$.  Let $\Aq$ and $\AQ$ be the monic polynomials over
$\F_q$ and $\F_{q^2}$, respectively.  We use
\[
 |f|=q^{\deg f}\quad(f\in\Aq),
 \qquad |F|_2=q^{2\deg F}\quad(F\in\AQ),
 \qquad \zeta_q(s)=\frac1{1-q^{1-s}}.
\]
Unless otherwise stated, $P$ denotes a monic irreducible polynomial
over $\F_q$.  The map $\sigma$ raises the coefficients of a polynomial
over $\F_{q^2}$ to the $q$th power.  All asymptotic statements keep $q$ fixed. We also use
$\zeta_Q(s)=(1-Q^{1-s})^{-1}$ and, in the polynomial variable,
$\zeta_Q(u)=(1-Qu)^{-1}$, as in \cite{Hong}; the argument specifies
which convention is intended.

Fix an identification of the fourth roots of unity in $\F_{q^2}$ with
those in $\C$.  For a monic irreducible $\pi\in\F_{q^2}[T]$, define
$\left(\frac{f}{\pi}\right)_4$ by the residue class of
$f^{(Q^{\deg\pi}-1)/4}$, and set it equal to zero if $\pi\mid f$.
Extend the symbol multiplicatively in its denominator.

For quartic characters, the conductor condition is
\begin{equation}\label{eq:conductor-condition}
 \cond(\chi^2)=\cond(\chi).
\end{equation}
For $m\ge1$, put $g=3m-3$ and define
\begin{equation}\label{eq:family}
 \Cfour=
 \left\{
 \chi_F=\left.\left(\frac{\cdot}{F}\right)_4\right|_{\F_q[T]}:
 F\in\AQ\text{ monic square-free},\ \deg F=m,\ (F,F^\sigma)=1
 \right\}.
\end{equation}
The coprimality condition is equivalent to requiring that no nonconstant
polynomial in
$\F_q[T]$ divides $F$.  It excludes both primes fixed by $\sigma$ and
the simultaneous occurrence of a pair $\pi,\pi^\sigma$.

\begin{lemma}\label{lem:family}
The parameterization \eqref{eq:family} is one-to-one onto the primitive
order-four characters satisfying \eqref{eq:conductor-condition} and
having conductor degree $2m$.  For these characters,
\[
 \cond(\chi_F)=FF^\sigma,\qquad
 \chi_{F^\sigma}=\overline{\chi_F}.
\]
They are even.  The associated cyclic quartic extension has genus
$g=3m-3$, and the completed $L$-polynomial has degree $2m-2$.
\end{lemma}
\begin{proof}
A base-field prime admits a local character of exact order four only
when its degree is even.  It then splits as $P=\pi\pi^\sigma$ over
$\F_{q^2}$, and the two choices $\pi$ and $\pi^\sigma$ give the two
local characters of exact order four.  Condition
\eqref{eq:conductor-condition} excludes local characters of order two
at ramified primes.  The Chinese remainder theorem now gives the claimed bijection; see
\cite[Lemma~3.6]{Lin}.

For a constant $a\in\F_q^\times$, its quartic residue symbol at a prime
of $\F_{q^2}[T]$ is one, since $q-1$ divides
$(q^{2d}-1)/4$.  Thus these characters are even.  Each finite ramified
prime has ramification index four and infinity is unramified.
Riemann--Hurwitz gives $2g-2=-8+3(2m)$, hence $g=3m-3$.
The degree of the completed Dirichlet $L$-polynomial is the conductor
degree minus two.
\end{proof}

\begin{remark}
The parameters $F$ and $F^\sigma$ give distinct conjugate characters.
The family \eqref{eq:family} includes only characters satisfying
\eqref{eq:conductor-condition}. Its genus parameter is $3m-3$,
whereas the completed $L$-polynomial has degree $2m-2$.
\end{remark}

\subsection{Counting the family}

\begin{lemma}\label{lem:count}
For every $\varepsilon>0$,
\begin{equation}\label{eq:count}
 |\Cfour|=c_q q^{2m}+O_{q,\varepsilon}(q^{(1+\varepsilon)m}),
\end{equation}
where
\begin{equation}\label{eq:cq}
 c_q=
 \prod_{\deg P\ \text{odd}}(1-|P|^{-2})
 \prod_{\deg P\ \text{even}}
 (1+2|P|^{-1})(1-|P|^{-1})^2>0.
\end{equation}
Moreover,
\begin{equation}\label{eq:cq-lower}
 c_q\ge\zeta_q(2)^{-3}=(1-q^{-1})^3.
\end{equation}
\end{lemma}
\begin{proof}
Write $\mu_q$ and $\mu_Q$ for the M\"obius functions over the two
fields. As in \cite[proof of Corollary~1.1]{Hong}, insert
\[
 1_{(F,F^\sigma)=1}=\sum_{\substack{D\in\Aq\\D\mid F}}\mu_q(D),
 \qquad \mu_Q^2(F)=\sum_{A^2\mid F}\mu_Q(A).
\]
For $|u|<q^{-2}$, the counting series, including $F=1$, is therefore
\begin{align}\label{eq:B}
 \sumF u^{\deg F}
 &=\sum_{D\in\Aq}\mu_q(D)u^{\deg D}
   \sum_{\substack{A,G\in\AQ\\(AG,D)=1}}
       \mu_Q(A)u^{2\deg A+\deg G}\notag\\
 &=\frac{\zeta_Q(u)}{\zeta_Q(u^2)}
   \prod_P\left(1-u^{\deg P}
       \prod_{\pi\mid P}(1+u^{\deg\pi})^{-1}\right)\notag\\
 &=\zeta_Q(u)
   \prod_{\deg P\text{ odd}}(1-u^{\deg P})
   \prod_{\deg P\text{ even}}
       (1+2u^{\deg P/2})(1-u^{\deg P/2})^2.
\end{align}
Here and below, $\pi$ denotes a monic prime over $\F_Q$.
The product multiplying $\zeta_Q(u)$ converges normally for
$|u|<q^{-1}$, since each local factor is $1+O(|u|^{\deg P})$.
Its value at $u=q^{-2}$ is $c_q$. Perron's formula and a shift to
$|u|=q^{-1-\varepsilon}$, with $0<\varepsilon<1$, give
\eqref{eq:count}. Finally, at each even-degree prime,
\[
 (1+2|P|^{-1})(1-|P|^{-1})^2\ge(1-|P|^{-2})^3;
\]
the same lower bound holds for the odd-degree factor. Multiplication
over $P$ gives \eqref{eq:cq-lower}.
\end{proof}

We use the coprimality factors $\rho_P$ and $\rho(H)$ from
\eqref{eq:rho}.
For a fixed $H$, inserting $(F,H)=1$ in the counting series removes
precisely the even-degree factors attached to $P\mid H$.  After multiplication by $1-q^2u$, its value at $u=q^{-2}$ is
$c_q\rho(H)$.

\subsection{The functional equation and Gauss sums}

For a primitive nonprincipal character, write
\[
 \Lpol(v,\chi)=\sum_{f\in\Aq}\chi(f)v^{\deg f},
 \qquad \Lq(s,\chi)=\Lpol(q^{-s},\chi).
\]
The polynomial $\Lpol(v,\chi)$ has a factor $1-v$ for the even
characters in \eqref{eq:family}.  With $|\omega(\chi)|=1$, its functional
equation is
\begin{equation}\label{eq:FE}
 \Lq(s,\chi)=
 \omega(\chi)q^{(m-1)(1-2s)}
 \frac{1-q^{-s}}{1-q^{s-1}}\Lq(1-s,\bar\chi).
\end{equation}
The reciprocal zeros of the completed $L$-polynomial have modulus
$q^{1/2}$. In particular, for fixed $\varepsilon>0$,
\begin{equation}\label{eq:lindelof}
 |\Lq(s,\chi)|\ll_{q,\varepsilon}q^{\varepsilon\deg\cond(\chi)}
 \qquad(\Rea s\ge1/2);
\end{equation}
see \cite[Lemmas~3.2--3.3]{Lin} for the bounds used in the generating
series arguments.

Choose the standard additive character determined by the coefficient of
$T^{-1}$ and define
\[
 G_Q(V,F)=\sum_{a\bmod F}\left(\frac{a}{F}\right)_4 e_Q(aV/F).
\]
When the arguments are coprime, the elementary identities are
\begin{align}
 G_Q(aV,F)&=\overline{\chi_F(a)}G_Q(V,F),\label{eq:Gshift}\\
 G_Q(V,AB)&=\chi_A(B)^2G_Q(V,A)G_Q(V,B).
 \label{eq:Gprod}
\end{align}
The reciprocity sign is one since $Q\equiv1\pmod8$. We apply
\cite[Theorem~5.3 and Corollary~5.4]{DFL25} over $\F_Q$.
Section~\ref{sec:first} establishes the root-number normalization
and treats the additional coprimality conditions.

\subsection{Exact averages of short polynomials}

For an arbitrary monic $R\in\AQ$, square-free or not, use $\chi_R$ for
the quartic residue symbol restricted to $\Aq$.

\begin{lemma}\label{lem:orthog}
Let $A,B\in\Aq$ be monic and suppose
$\deg\rad(AB)\le m$.  Then
\begin{equation}\label{eq:orthog}
 q^{-2m}\sum_{\substack{R\in\AQ\\\deg R=m}}
 \chi_R(A)\overline{\chi_R(B)}
 =\begin{cases}
 \displaystyle\frac{\phi_Q(\rad(AB))}{|\rad(AB)|_2},
       &AB^3\text{ is a fourth power},\\[2mm]
 0,&\text{otherwise}.
 \end{cases}
\end{equation}
The radical and Euler function on the right may be evaluated after
factorization over $\F_{q^2}$.
\end{lemma}
\begin{proof}
By quartic reciprocity, the summand is the character
$\left(\frac{R}{AB^3}\right)_4$.  It is induced modulo $\rad(AB)$, and the degree
condition makes monic degree-$m$ polynomials equidistributed over its
residue classes.  The induced character is principal exactly when
all the exponents in $AB^3$ are divisible by four.  Its principal
average is the proportion of invertible classes; otherwise its
average is zero.
\end{proof}

Let $(X_P)_P$ be independent random variables with
\begin{equation}\label{eq:model}
 \Pr(X_P=0)=1-\gamma_P,\qquad
 \Pr(X_P=\xi)=\frac{\gamma_P}{4}\quad(\xi^4=1),
\end{equation}
where
\begin{equation}\label{eq:gamma}
 \gamma_P=\begin{cases}
 1-|P|^{-2},&\deg P\text{ odd},\\
 (1-|P|^{-1})^2,&\deg P\text{ even}.
 \end{cases}
\end{equation}
By Lemma~\ref{lem:orthog}, the average over all monic $R$ of degree
$m$ of a polynomial in $\chi_R(P)$ and their conjugates equals its
expectation under \eqref{eq:model}, provided each monomial satisfies
the stated degree condition. For nonnegative polynomials this
average bounds the corresponding normalized sum over admissible $F$.

\begin{lemma}\label{lem:model-bounds}
If $Z=\sum_Pc_PX_P$ and $V=\sum_P|c_P|^2$, then, for $t\ge0$ and
$p\ge2$,
\[
 \E e^{t\Rea Z}\le e^{t^2V/2},\qquad
 \norm{Z}_p\le4\sqrt{pV},\qquad
 \E e^{t|Z|}\le4e^{t^2V}.
\]
\end{lemma}
\begin{proof}
Each real summand has mean zero and is bounded by $|c_P|$.
Hoeffding's lemma and independence give the exponential estimate.  Its tail bound, applied to the real and
imaginary parts, gives the displayed moment estimate.  Finally
$|Z|\le|\Rea Z|+|\Im Z|$ and summing over the four choices of signs gives
the last inequality.
\end{proof}
The degree condition in Lemma~\ref{lem:orthog} is imposed whenever these
estimates are transferred to character averages.

\section{Twisted first moment}\label{sec:first}

\subsection{The two-variable series and the fourth-power condition}

For a monic twist $h$, write uniquely
\begin{equation}\label{eq:hstar}
 h=C_1C_2^2C_3^3E^4,\qquad h^\ast=C_1^3C_2^2C_3,
\end{equation}
where the $C_i$ are square-free and pairwise coprime; no coprimality
between $E$ and the $C_i$ is imposed.  Thus $hh^\ast$ is a fourth power
and $|h^\ast|\le|h|^3$.  For $(F,h)=1$, we have $\chi_F(h)=\bar\chi_F(h^\ast)$.

For $\Rea s,\Rea w$ sufficiently large, define
\begin{equation}\label{eq:Afirst}
 A_4(s,w;h)=
 \sum_{\substack{F\in\AQ\text{ square-free}\\
 (F,F^\sigma)=1,\ \deg F\ge1}}
 \frac{\chi_F(h)\Lq(w,\chi_F)}{|F|_2^s}.
\end{equation}
With $u=q^{-2s}$, use $A_4(u,w;h)$ for the same function.  Coefficient
extraction gives the exact identity
\begin{equation}\label{eq:perron-first}
 \sumchi\chi(h)\Lq(1/2,\chi)
 =\frac1{2\pi i}\oint_{|u|=r}
 A_4(u,1/2;h)\frac{du}{u^{m+1}},
\end{equation}
initially for a sufficiently small circle.

For $H\in\Aq$, write $\chi^{(H)}(R)=\left(\frac{R}{H}\right)_4$ over $\F_Q[T]$.
Expanding the $L$-function and applying the two
M\"obius identities used in \eqref{eq:B}, quartic reciprocity gives
\begin{align}\label{eq:Aexpand}
 A_4(s,w;h)
 &=\sum_{f\in\Aq}\frac1{|f|^w}\Bigg\{
 \frac{L_Q(s,\chi^{(hf)})}{L_Q(2s,(\chi^{(hf)})^2)}\notag\\
 &\qquad\times\sum_{\substack{D\in\Aq\\(D,hf)=1}}
 \frac{\mu_q(D)}{|D|_2^s}
 \prod_{\pi\mid D}
 \left(1+\frac{\chi_\pi(hf)}{|\pi|_2^s}\right)^{-1}
 -1\Bigg\}.
\end{align}
The identity holds initially for $\Rea s,\Rea w>1$.
The subtraction removes $F=1$.
The diagonal condition
is
\begin{equation}\label{eq:first-diagonal}
 hf\text{ is a fourth power}\quad\Longleftrightarrow\quad
 f=h^\ast B^4\quad(B\in\Aq).
\end{equation}

\subsection{The arithmetic factor}

\begin{proposition}\label{prop:first-factor}
The contribution of the simple pole at $s=1$ from the diagonal in
\eqref{eq:Aexpand}, evaluated at $w=1/2$, is $q^{2m}\mathcal M_{4,q}(h)$,
where
\begin{equation}\label{eq:Mmain}
 \mathcal M_{4,q}(h)=\frac{c_q}{\sqrt{|h^\ast|}}
 \sum_{B\in\Aq}\frac{\rho(hB)}{|B|^2}.
\end{equation}
The series converges absolutely and has the Euler product
\begin{align}\label{eq:Mmain-product}
 \mathcal M_{4,q}(h)
 ={}&\frac{c_q\zeta_q(2)}{\sqrt{|h^\ast|}}
 \prod_P\left(1-(1-\rho_P)|P|^{-2}\right)\notag\\
 &\times\prod_{P\mid h}
 \frac{\rho_P}{1-(1-\rho_P)|P|^{-2}}.
\end{align}
\end{proposition}
\begin{proof}
If $hf$ is a fourth power, $\chi_F(hf)$ is the indicator of
$(F,hf)=1$.  By Lemma~\ref{lem:count} and \eqref{eq:rho}, its generating series, multiplied by $1-q^{2-2s}$,
has value $c_q\rho(hf)$ at $s=1$.  Substitution of
\eqref{eq:first-diagonal} gives $\rho(hf)=\rho(hB)$ and the weight
$|h^\ast|^{-1/2}|B|^{-2}$.  At $P\nmid h$, the local $B$-sum is
\[
 1+\frac{\rho_P|P|^{-2}}{1-|P|^{-2}}
 =\frac{1-(1-\rho_P)|P|^{-2}}{1-|P|^{-2}},
\]
whereas at $P\mid h$ it is $\frac{\rho_P}{1-|P|^{-2}}$.
Multiplying the local factors proves \eqref{eq:Mmain-product}.
\end{proof}

\subsection{The base-field root number}

Put
\[
 \kappa_q=(-1)^{(q+1)/4}.
\]
\begin{lemma}\label{lem:root-normalization}
For an admissible $F$ of degree $d\ge1$, with the additive characters
specified in Section~\ref{sec:prelim},
\begin{equation}\label{eq:root-normalization}
 G_q(\chi_F)=\chi_F(F^\sigma)G_Q(1,F)
 =\kappa_q^dG_Q(1,F),\qquad
 \omega(\chi_F)=q^{-d}\kappa_q^dG_Q(1,F).
\end{equation}
Here $G_q(\chi_F)$ has modulus $FF^\sigma$, whereas $G_Q(1,F)$
has modulus $F$.
\end{lemma}
\begin{proof}
The map
\[
 b\pmod F\longmapsto bF^\sigma+b^\sigma F\pmod{FF^\sigma}
\]
is a bijection onto the base-field residue classes: it is invariant
under $\sigma$, and is invertible by the Chinese remainder theorem.
Its additive character is $e_Q(b/F)$, while its multiplicative
character is $\chi_F(b)\chi_F(F^\sigma)$. This proves the first equality.
If $R=\operatorname{Res}_T(F,F^\sigma)$, the norm definition of the
quartic symbol gives $\chi_F(F^\sigma)=R^{(Q-1)/4}$, interpreted through
our fixed identification of roots of unity. Moreover
\[
 R^q=\operatorname{Res}_T(F^\sigma,F)=(-1)^dR.
\]
Consequently $R^{(Q-1)/4}=(-1)^{d(q+1)/4}$.
Finally, the root-number formula for a primitive even character of
conductor degree $2d$ is $q^{-d}G_q(\chi_F)$.
\end{proof}

\subsection{Gauss series with the additional conductor restrictions}

For monic $r\in\AQ$ and $h\in\Aq$, define
\[
 \widetilde\Psi_Q(r,v;h)=\sum_{\substack{F\in\AQ\\(F,rh)=1}}G_Q(r,F)v^{\deg F}.
\]
If $(F,rh)=1$, then $G_Q(r,F)$ vanishes unless $F$ is square-free,
and otherwise $|G_Q(r,F)|=|F|_2^{1/2}$. Hence, for $|v|<Q^{-3/2}$,
\[
 |\widetilde\Psi_Q(r,v;h)|\le\zeta_Q(Q^{1/2}|v|).
\]
\begin{lemma}\label{lem:restricted-gauss}
Let $5/4<\alpha<3/2$ and $|v|=Q^{-\alpha}$. On closed
subintervals of this range, for every $\epsilon>0$,
\begin{equation}\label{eq:needed-gauss}
 \widetilde\Psi_Q(r,v;h)\ll_{q,\epsilon,\alpha}
 |rh|_2^\epsilon |r|_2^{3/4-\alpha/2}.
\end{equation}
The implied constant is uniform in $r,h$ and in the argument of $v$.
The series has a holomorphic continuation throughout
$|v|<Q^{-5/4}$.
\end{lemma}
\begin{proof}
Let $a$ be the product of the extension-field primes dividing $h$
but not $r$. In the notation of \cite[Section~5]{DFL25}, set
\[
 \psi_r^{(i)}(r,E,v)=\frac1{1-Q^4v^4}
 \sum_{\substack{F\in\AQ\\(F,r)=1,\ E\mid F\\
                    \deg F\equiv i\ (4)}}G_Q(r,F)v^{\deg F}.
\]
Finite inclusion--exclusion gives the exact identity
\begin{equation}\label{eq:restricted-inclusion}
 \widetilde\Psi_Q(r,v;h)=(1-Q^4v^4)
 \sum_{E\mid a}\mu_Q(E)\sum_{i=0}^3\psi_r^{(i)}(r,E,v).
\end{equation}
Theorem~5.3 and Corollary~5.4 of \cite{DFL25}, with base field $Q$
and order four, apply since $E$ is square-free and $(E,r)=1$.
More explicitly, write
$r=r_1r_2^2r_3^3r_4^4$, with $r_1,r_2,r_3$ square-free and pairwise
coprime, as required by the cited theorem. Its second case applies
for order $\ell=4$: throughout $1<\alpha<3/2$ one has
$4\le6/(3-2\alpha)$. The exponent of $|E|_2$ becomes
$\ell(3/4-\alpha/2)-1=2-2\alpha$, while the exponent of $|r|_2$
is $3/4-\alpha/2$. Consequently the quoted bound is
\[
 \psi_r^{(i)}(r,E,v)
 \ll |Er|_2^\epsilon |E|_2^{2-2\alpha}
                   |r|_2^{3/4-\alpha/2}.
\]
The divisor sum is bounded by
\[
 |r|_2^\epsilon\sum_{E\mid a}|E|_2^{2-2\alpha+\epsilon}
 =|r|_2^\epsilon\prod_{\pi\mid a}
                (1+|\pi|_2^{2-2\alpha+\epsilon}).
\]
Choose the preliminary $\epsilon$ smaller than $2\alpha-2$ on
the chosen closed interval. Each factor is then at most two, and
$2^{\omega(a)}\ll_{Q,\epsilon'}|a|_2^{\epsilon'}$ for every
$\epsilon'>0$: separate the finitely many primes of bounded degree
and use $2\le|\pi|_2^{\epsilon'}$ for the others.
Since $a\mid\operatorname{rad}(h)$, this proves the asserted bound
after replacing the small exponents by $\epsilon$.
The unrestricted metaplectic series satisfies
\[
 \psi^{(i)}(r,v)=\frac{v^iP_{r,i}(v^4)}{1-Q^5v^4}
\]
with $P_{r,i}$ a polynomial; see \cite[equation~(4.2)]{DFL25}.
The finite decomposition in Theorem~5.3 expresses the restricted
series through such functions and finite local factors. Its possible
poles are zeros of $1-Q^5v^4$ or of the finite denominator factors
in \cite[Theorem~5.3]{DFL25}. In the latter the quartic reciprocity
sign is one because
$Q\equiv1\pmod8$; their form is $1-Q^{3\deg\pi}v^{4\deg\pi}$.
The global factor has zeros at $|v|=Q^{-5/4}$ and the local factors
at $|v|=Q^{-3/4}$. The decomposition and \eqref{eq:restricted-inclusion} therefore give
holomorphy for $|v|<Q^{-5/4}$, with uniform bounds on compact
subannuli.
\end{proof}

\subsection{The first convergence region}

Write $\sigma_0=\Rea s$, $\tau=\Rea w$, $u=q^{-2s}$ and
$v=q^{-w}$; for $S\subset\mathbb R^2$, put $T(S)=S+i\mathbb R^2$.
We bound the factors in \eqref{eq:Aexpand} separately.
For $\sigma_0\ge1/2+\delta$ and any $\eta>0$, Lindel\"of and
absolute convergence of the reciprocal Euler product give
\begin{equation}\label{eq:Sone-bound}
 (1-Q^{1-s})
 \frac{L_Q(s,\chi^{(hf)})}{L_Q(2s,(\chi^{(hf)})^2)}
 \ll_{q,\delta,\eta}|hf|^\eta.
\end{equation}
The character $\chi^{(hf)}$ is principal precisely when $hf$ is a
fourth power. In that case its numerator is
\[
 L_Q(s,\chi^{(hf)})=\zeta_Q(s)
       \prod_{\pi\mid hf}(1-|\pi|_2^{-s}),
\]
and the factor $1-Q^{1-s}$ removes its pole. The finite Euler
factors of imprimitive characters are $O_{q,\delta,\eta}(|hf|^\eta)$.
Moreover,
\[
 \prod_{\pi\mid D}
 \left|1+\chi_\pi(hf)|\pi|_2^{-s}\right|^{-1}
 \le\prod_{\pi\mid D}(1-|\pi|_2^{-\sigma_0})^{-1}
 \ll_{q,\delta,\eta}|D|^\eta.
\]
Consequently the $D$-sum is bounded by
$O_{q,\delta,\eta}(\zeta_q(2\sigma_0-\eta))$. Inserting these
estimates into \eqref{eq:Aexpand} yields
\[
 (1-Q^{1-s})A_4(s,w;h)
 \ll_{q,\delta,\eta}|h|^\eta
       \zeta_q(\tau-\eta)\zeta_q(2\sigma_0-\eta).
\]
On every compact subregion with $\sigma_0>1/2$ and $\tau>1$,
choose $\eta<\min(2\sigma_0-1,\tau-1)$ uniformly. The expanded
series then converges normally. Thus $(1-Q^{1-s})A_4(s,w;h)$ is
holomorphic on $T(S_1)$, where
\begin{equation}\label{eq:Sone}
 S_1=\{(\sigma_0,\tau):\sigma_0>1/2,\ \tau>1\},
\end{equation}
and is $O_{q,\eta}(|h|^\eta)$ on compact real subregions, uniformly
in the imaginary parts.

\subsection{The residue at \texorpdfstring{$s=1$}{s=1}}

For $\tau>1$, the normal convergence just proved allows evaluation
of $\left[(1-Q^{1-s})A_4(s,w;h)\right]_{s=1}$ term by term. The nonprincipal terms vanish
there after multiplication by $1-q^{2-2s}$. If $hf$ is a fourth
power, the local counting calculation in
Proposition~\ref{prop:first-factor} instead gives $c_q\rho(hf)$.
Using $f=h^\ast B^4$ we obtain
\begin{equation}\label{eq:residue-function}
 \left[(1-Q^{1-s})A_4(s,w;h)\right]_{s=1}
 =c_q|h^\ast|^{-w}\sum_{B\in\Aq}\rho(hB)|B|^{-4w}.
\end{equation}
The expression on the right converges normally for $\tau>1/4$,
because $0<\rho\le1$ and
$\sum_{\deg B=n}|B|^{-4\tau}=q^{(1-4\tau)n}$.
At $w=1/2$, the value is $\mathcal M_{4,q}(h)$.

\subsection{The second convergence region}

An admissible conductor parameter of degree $d$ has conductor degree
$2d$, and there are at most $Q^d=q^{2d}$ such parameters. The
Lindel\"of estimate and the functional equation yield, on fixed
bounded real $w$-intervals,
\begin{equation}\label{eq:family-L-bound}
 |L_q(w,\chi_F)|\ll_{q,\eta}
 q^{2d\max(0,1/2-\tau)+\eta d}.
\end{equation}
For $\tau\ge1/2$, this follows from \eqref{eq:lindelof}.
For $\tau<1/2$, apply
\eqref{eq:FE} and Lindel\"of at $1-w$; its power of $q$ is
$q^{(d-1)(1-2\tau)}$. The remaining quotient is bounded on closed
subregions of $\tau<1/2$. All bounds are uniform in $\Im w$.

Since $|\chi_F(h)|\le1$, summing \eqref{eq:Afirst} by conductor
degree gives
\[
 |A_4(s,w;h)|\ll_{q,\eta}
 \sum_{d\ge1}q^{[2-2\sigma_0+2\max(0,1/2-\tau)+\eta]d}
 \le\zeta_Q\bigl(\sigma_0-\max(0,1/2-\tau)-\eta/2\bigr).
\]
For sufficiently small $\eta$, this gives normal convergence in
\begin{align}\label{eq:Stwo}
 S_2
 &=\{\sigma_0>1+\max(0,1/2-\tau)\}\notag\\
 &=\{\sigma_0>1,\ \sigma_0+\tau>3/2\}.
\end{align}
Hence $A_4$ and $(1-Q^{1-s})A_4(s,w;h)$ are holomorphic on $T(S_2)$, with
bounds independent of $h$.

\subsection{The third representation}

In this calculation write $\alpha=\sigma_0+\tau$ and $r=fh^\ast$.
For $\tau<0$ and $\alpha>3/2$ all expansions in the following
calculation converge absolutely. Combining the degree-$d$
functional equation with Lemma~\ref{lem:root-normalization} gives
\[
 |F|_2^{-s}q^{(d-1)(1-2w)}q^{-d}
 =q^{2w-1}|F|_2^{-(s+w)}.
\]
Also, including vanishing at noncoprime arguments,
\[
 \chi_F(h)\bar\chi_F(f)
 =1_{(F,h)=1}\bar\chi_F(fh^\ast).
\]
It follows that
\begin{equation}\label{eq:dual-before-mobius}
 A_4(s,w;h)=q^{2w-1}\frac{1-q^{-w}}{1-q^{w-1}}\sum_{f\in\Aq}|f|^{w-1}
 \sum_{\substack{F\text{ admissible},\ \deg F\ge1\\(F,fh)=1}}
 \kappa_q^{\deg F}G_Q(r,F)|F|_2^{-s-w}.
\end{equation}
For a square-free extension-field polynomial, admissibility is
expressed by the exact base-field M\"obius identity
\[
 1_{(F,F^\sigma)=1}
 =\sum_{\substack{D\in\Aq\\D\mid F}}\mu_q(D).
\]
The divisor sum vanishes whenever a nonconstant base-field polynomial
divides $F$, and equals one otherwise.

Insert this identity and write $F=DE$. The square-free condition
implies $(D,E)=1$, and the twist additionally requires $(D,fh)=1$
and $(E,fh)=1$. Quartic reciprocity and Gauss multiplicativity give
\begin{align*}
 \bar\chi_{DE}(r)G_Q(1,DE)
 &=\bar\chi_D(r)\bar\chi_E(r)\chi_D(E)^2
       G_Q(1,D)G_Q(1,E)\\
 &=G_Q(r,D)G_Q(rD^2,E).
\end{align*}
In the last step $\bar\chi_E(D^2)=\chi_E(D)^2$ and
$\chi_D(E)=\chi_E(D)$. Since $(E,rD^2h)=1$, a nonsquare-free
$E$ contributes zero; this follows from the vanishing of a local
Gauss sum with a unit argument and a prime-power modulus of exponent
at least two. We may therefore remove the square-free condition
on $E$. This proves the exact identity
\begin{align}\label{eq:dual-exact}
 A_4(s,w;h)
 ={}&q^{2w-1}\frac{1-q^{-w}}{1-q^{w-1}}\Bigg\{
 \sum_{f\in\Aq}\frac1{|f|^{1-w}}
 \sum_{\substack{D\in\Aq\\(D,fh)=1}}
 \frac{\mu_q(D)\kappa_q^{\deg D}G_Q(r,D)}{|D|_2^{s+w}}\notag\\
 &\hspace{23mm}\times
 \widetilde\Psi_Q(rD^2,\kappa_q Q^{-s-w};h)
 -\zeta_q(1-w)\Bigg\}.
\end{align}
The extra restriction from primes dividing the fourth-power part of
$h$ remains inside the Gauss series even when those primes do not divide
$h^\ast$. The term $F=1$ introduced by M\"obius inversion has
$D=E=1$, contributes one for each $f$, and hence accounts for the
subtraction $\zeta_q(1-w)$.

\subsection{The third convergence region}

First assume $5/4<\alpha<3/2$. Lemma~\ref{lem:restricted-gauss}
and $|G_Q(r,D)|\le |D|_2^{1/2}=|D|$ give
\begin{align*}
 |\widetilde\Psi_Q(rD^2,\kappa_qQ^{-s-w};h)|
 &\ll |h|^\eta|rD^2|_2^{3/4-\alpha/2+\eta/2}\\
 &\ll |h|^\eta
       |f|^{3/2-\alpha+\eta}
       |h^\ast|^{3/2-\alpha+\eta}
       |D|^{3-2\alpha+2\eta}.
\end{align*}
Here $|rD^2|_2=|f|^2|h^\ast|^2|D|^4$, and the small exponents
have been relabeled. Substitution into \eqref{eq:dual-exact} gives, on compact real
subregions with $\tau<0$,
\begin{align}\label{eq:majorant-first}
 |A_4(s,w;h)|
 &\ll |h|^\eta|h^\ast|^{3/2-\alpha+\eta}
 \sum_{f\in\Aq}|f|^{1/2-\sigma_0+\eta}
 \sum_{D\in\Aq}|D|^{4-4\alpha+2\eta}+1\notag\\
 &=|h|^\eta|h^\ast|^{3/2-\alpha+\eta}
 \zeta_q(\sigma_0-1/2-\eta)\zeta_q(4\alpha-4-2\eta)+1.
\end{align}
Thus the two sums converge normally when
$\sigma_0>3/2$ and $\alpha>5/4$, after choosing $\eta$ smaller
than the convergence margins. The prefactor and the subtracted
$\zeta_q(1-w)$ are holomorphic for $\tau<0$.

For $\alpha>3/2$ we can instead use the original Gauss series:
\[
 |\widetilde\Psi_Q(R,\kappa_qQ^{-s-w};h)|
 \le\zeta_Q(\alpha-1/2).
\]
The two remaining sums are bounded by
$\zeta_q(1-\tau)$ and $\zeta_q(2\alpha-1)$.
They converge for $\tau<0$ and $\alpha>3/2$, independently of
$r,h$. To justify a neighborhood of the joining line $\alpha=3/2$,
fix a point with $\tau<0$ and choose small $\nu>0$ with
$2\nu<-\tau$ and $4\nu<1$. By the maximum-modulus principle in
the holomorphic disk of Lemma~\ref{lem:restricted-gauss}, the bound
on $|v|=Q^{-3/2+\nu}$ also bounds smaller circles:
\[
 \widetilde\Psi_Q(R,v;h)\ll |h|^\eta|R|_2^{\nu/2+\eta}
 \quad(|v|\le Q^{-3/2+\nu}).
\]
The resulting degree exponents near the joining line are
$\tau+\nu+O(\eta)<0$ and
$2-2\alpha+2\nu+O(\eta)<0$. Both are negative on a neighborhood of the fixed point, for
sufficiently small $\eta$.

Thus $(1-Q^{1-s})A_4(s,w;h)$ is holomorphic on $T(S_3)$, where
\begin{equation}\label{eq:Sthree}
 S_3=\{\sigma_0>3/2,\ \sigma_0+\tau>5/4,\ \tau<0\}.
\end{equation}
For $5/4<\alpha<3/2$, we have
\begin{equation}\label{eq:Sthree-bound}
 (1-Q^{1-s})A_4(s,w;h)\ll
 |h|^\eta|h^\ast|^{3/2-\sigma_0-\tau+\eta},
\end{equation}
on compact real subregions, uniformly in imaginary parts. The
subtracted zeta term is bounded there and is absorbed since the
right side is at least one. 

\subsection{Analytic continuation}

The convex hull of the three regions is
\begin{equation}\label{eq:first-hull}
 \operatorname{conv}(S_1\cup S_2\cup S_3)
 =\{\sigma_0>1/2,\ \sigma_0+\tau>5/4,
                        \ 5\sigma_0+4\tau>13/2\}.
\end{equation}

The character and family expansions agree on $T(S_1\cap S_2)$
by their common initially absolutely convergent expansion. The dual
representation agrees with the family expansion on the nonempty
part $\tau<0$, $\sigma_0+\tau>3/2$ of $T(S_2\cap S_3)$.
The identity theorem extends these identities across the connected
overlaps. Hence $(1-Q^{1-s})A_4(s,w;h)$ is holomorphic on the connected
tube $T(S_1\cup S_2\cup S_3)$. Bochner's tube theorem, in the form
recalled in \cite[Section~2.3]{Hong}, continues it holomorphically
to $T(\operatorname{conv}(S_1\cup S_2\cup S_3))$.

In terms of $u=q^{-2s}$ and $v=q^{-w}$, this region is
\begin{equation}\label{eq:hull-moduli}
 |u|<q^{-1},\qquad |u|^{1/2}|v|<q^{-5/4},\qquad
 |u|^5|v|^8<q^{-13}.
\end{equation}
Its intersection with $\tau=1/2$ is
\begin{equation}\label{eq:central-slice}
 \{(\sigma_0,1/2):\sigma_0>9/10\}.
\end{equation}

The zeros of $1-Q^5(\kappa_qQ^{-s-w})^4$ satisfy
$\sigma_0+\tau=5/4$, outside
$\operatorname{conv}(S_1\cup S_2\cup S_3)$.
Thus the continued $A_4(s,w;h)$ has no poles in this region except
possibly at $s=1+\pi i k/\log q$, arising from $(1-Q^{1-s})^{-1}$.
These points correspond to $u=q^{-2}$.

\subsection{Uniformity in the twist}

The bounds on $S_1$ and $S_3$ also control the dependence on $h$
after continuation. Apply the three-lines theorem to
$(1-Q^{1-s})A_4(s,w;h)$ along complex lines joining the interiors
of these two regions. The endpoint bounds are $O(|h|^\eta)$ and
$O(|h|^\eta|h^\ast|^{1/4-2\delta+\eta})$, respectively, at real
points $(1/2+\delta,1+\delta)$ and
$(3/2+\delta,-1/4+\delta)$. Thus
\begin{equation}\label{eq:interpolation}
 (1-Q^{1-s})A_4(s,1/2;h)
 \ll_{q,\delta,\eta}|h|^\eta|h^\ast|^{1/10+\eta},
 \qquad \Rea s=\frac9{10}+\frac{9\delta}{5}.
\end{equation}
For fixed $h$, holomorphy and the periods $\pi/\log q$ in $s$ and
$2\pi/\log q$ in $w$ give the boundedness required on the
interpolation strips. Choose $\delta=5\varepsilon/18$ and then
$\eta$ sufficiently small. Since $1-Q^{1-s}$ is bounded away from
zero on the resulting contour, we obtain
\begin{equation}\label{eq:uniform-contour}
 A_4(s,1/2;h)\ll_{q,\varepsilon}
 |h^\ast|^{1/10+\varepsilon}|h|^\varepsilon,
 \qquad \Rea s=\frac9{10}+\frac\varepsilon2.
\end{equation}
The estimate is uniform in every monic $h$ and in $\Im s$.

\subsection{The twisted first moment}

\begin{theorem}[Uniform twisted first moment]\label{input:first}
For every $\varepsilon>0$, uniformly in monic $h$ and integers $m\ge1$,
\begin{align}\label{eq:first-input}
 \sumchi\chi(h)\Lq(1/2,\chi)
 ={}&q^{2m}\mathcal M_{4,q}(h)\notag\\
 &+O_{q,\varepsilon}\left(q^{(9/5+\varepsilon)m}
            |h^\ast|^{1/10+\varepsilon}|h|^\varepsilon\right),
\end{align}
where $g=3m-3$ and $\mathcal M_{4,q}(h)$ is the absolutely convergent
factor in \eqref{eq:Mmain}--\eqref{eq:Mmain-product}.
\end{theorem}
\begin{proof}
At $s=1$, the continuation region contains $\tau>3/8$.
Thus $\left[(1-Q^{1-s})A_4(s,w;h)\right]_{s=1}$ is holomorphic on that half-plane and
agrees, for $\tau>1$, with \eqref{eq:residue-function}.
The latter expression is holomorphic for $\tau>1/4$; the one-variable
identity theorem shows that they agree throughout $\tau>3/8$.
In particular, $[(1-Q^{1-s})A_4(s,1/2;h)]_{s=1}=\mathcal M_{4,q}(h)$.

By periodicity in $s$ and \eqref{eq:central-slice},
$A_4(u,1/2;h)$ extends meromorphically to $|u|<q^{-9/5}$.
Since $\left[(1-Q^{1-s})A_4(s,1/2;h)\right]_{s=1}=\mathcal M_{4,q}(h)>0$, its unique pole
in this disk is simple and lies at $u=q^{-2}$. Hence
\[
 \Res_{u=q^{-2}}\frac{A_4(u,1/2;h)}{u^{m+1}}
 =-q^{2m}\mathcal M_{4,q}(h).
\]
Fix $0<\varepsilon<1/10$ and set $R=q^{-9/5-\varepsilon}$.
Since $q^{-2}<R<q^{-9/5}$, the residue theorem applied to
\eqref{eq:perron-first} gives
\[
 \sumchi\chi(h)\Lq(1/2,\chi)
 =q^{2m}\mathcal M_{4,q}(h)+
   \frac1{2\pi i}\oint_{|u|=R}
                  A_4(u,1/2;h)\frac{du}{u^{m+1}}.
\]
By \eqref{eq:uniform-contour}, the absolute value of the last integral
is at most
\[
 R^{-m}\sup_{|u|=R}|A_4(u,1/2;h)|
 \ll_{q,\varepsilon}q^{(9/5+\varepsilon)m}
                    |h^\ast|^{1/10+\varepsilon}|h|^\varepsilon.
\]
This proves \eqref{eq:first-input}.
\end{proof}

\begin{remark}\label{rem:first-status}
If $|b(h)|\le1$ and $\deg h\le L$, then
$|h^\ast|\le|h|^3$ and summing the error against
$b(h)|h|^{-1/2}$ gives, after renaming the small exponents,
\[
 q^{(9/5+\varepsilon)m}
 \sum_{n=0}^L q^n q^{(-1/2+3/10+\varepsilon)n}
 \ll_{q,\varepsilon}
 q^{(9/5+\varepsilon)m+(4/5+\varepsilon)L}.
\]
Hence $L<(1/4-\delta)m$, for a fixed $\delta>0$ and sufficiently
small $\varepsilon$, gives a power saving compared with $q^{2m}$.
\end{remark}

\section{Mollifier}\label{sec:second}\label{sec:GZ}

We construct a mollifier from squared truncated exponentials and
decompose the family into good and exceptional sets, following the
framework of Curran~\cite{Curran24} and its Dirichlet analogue in
Gao and Zhao~\cite[Sections~2.7, 3 and~5.9]{GZ}.
The twisted first moment is supplied by Theorem~\ref{input:first},
and character averages are evaluated by Lemma~\ref{lem:orthog}.

\subsection{Parameters, prime squares, and the mollifier}

All logarithms without a subscript are natural. The following parameters
are used for the quartic, sextic, and cubic families throughout the paper.
For sufficiently large $m$, set
\begin{align}\label{eq:GZparameters}
 \theta_j&=e^j(\log m)^{-2},&
 J&=\max\{j\ge0:\theta_j\le e^{-51}\},\notag\\
 N_j&=\lfloor2m\theta_j\rfloor,&
 K_j&=\theta_j^{-3/4},&\ell_j&=2\lceil50K_j\rceil.
\end{align}
Thus $e^{-52}<\theta_J\le e^{-51}$ and $J=O(\log\log m)$.
Put $B=\lceil4\log_q(2m)\rceil$; eventually $2B<N_0$. Write
$I_0=(0,N_0]$, $I_j=(N_{j-1},N_j]$ for $j\ge1$, and
\[
 a_N(P)=e^{-\deg P/N}(1-\deg P/N),\qquad
 Q_B(\chi)=\frac12\sum_{\deg P\le B}\frac{\chi(P)^2}{|P|}.
\]
For $0\le r\le u\le J$ define
\begin{equation}\label{eq:GZblocks}
 \mathcal P_{r,N_u}(\chi)=
 \sum_{\deg P\in I_r}\frac{a_{N_u}(P)\chi(P)}{\sqrt{|P|}}
 +1_{\{r=0\}}Q_B(\chi).
\end{equation}
The square term is the same at every scale. Define
\begin{equation}\label{eq:GZmollifier}
 E_\ell(z)=\sum_{v=0}^\ell\frac{z^v}{v!},\qquad
 M_G(\chi)=\prod_{r=0}^J
 E_{\ell_r}\!\left(-\frac{\mathcal P_{r,N_J}(\chi)}2\right)^2.
\end{equation}
Then
$|M_G|=\prod_r|E_{\ell_r}(-\mathcal P_{r,N_J}/2)|^2$ is a
nonnegative finite polynomial in characters and their conjugates.

\begin{lemma}\label{lem:GZlength}
Writing $M_G(\chi)=\sum_{\deg h\le L}b_m(h)\chi(h)/\sqrt{|h|}$,
one has $|b_m(h)|\le1$ and may take
\begin{equation}\label{eq:GZlength}
 L=2\sum_{r=0}^J\ell_rN_r,
 \qquad \frac Lm\le
 \frac{400\theta_J^{1/4}}{1-e^{-1/4}}+
 \frac{8\theta_J}{1-e^{-1}}<0.005248338.
\end{equation}
\end{lemma}
\begin{proof}
Each term in the first block has degree at most $N_0$, including
$P^2$ since $2B<N_0$. Expansion of the squares proves the degree bound.
For absolute coefficients, replace all signs by plus signs and remove
the Taylor cutoffs. At each prime this is coefficientwise bounded by
$e^{x+x^2/2}=\sum_{v\ge0}\alpha_vx^v$. Here
$\alpha_0=\alpha_1=\alpha_2=1$ and
$v\alpha_v=\alpha_{v-1}+\alpha_{v-2}$, so $0\le\alpha_v\le1$.
Consequently $|b_m(h)|\le\prod_{P^v\Vert h}\alpha_v\le1$.
Finally $\ell_r\le100\theta_r^{-3/4}+2$ and summing the two
geometric series gives \eqref{eq:GZlength}.
\end{proof}

\begin{lemma}[Short logarithmic upper bound]\label{lem:GZshort}
Uniformly in $0\le j\le J$, for sufficiently large $m$,
\begin{equation}\label{eq:GZshort}
 \log|\Lq(1/2,\chi)|\le
 \Rea\sum_{r=0}^j\mathcal P_{r,N_j}(\chi)+\theta_j^{-1}+2.
\end{equation}
At a zero, the left side is interpreted as $-\infty$.
\end{lemma}
\begin{proof}
The logarithmic-derivative argument of
\cite[Lemma~3.1]{DFL21}, before its prime-power terms are separated,
gives, with conductor degree $2m$,
\begin{equation}\label{eq:log-upper}
 \log|\Lq(1/2,\chi)|\le
 \Rea\sum_{\deg f\le N}
 \frac{\Lambda(f)\chi(f)}{\deg f\sqrt{|f|}}
 e^{-\deg f/N}(1-\deg f/N)+\frac{2m}{N}.
\end{equation}
The logarithmic-derivative argument uses the shift $1/(N\log q)$
and the fact that the reciprocal zeros of the completed $L$-polynomial
have modulus $q^{1/2}$; the zero contribution is then nonpositive. For the
even characters considered here, the factor $1-q^{-s}$ is retained.
The terms of exponent at least three have absolute sum at most
\[
 \sum_P\sum_{v\ge3}\frac1{v|P|^{v/2}}
 \le\frac{-\log(1-q^{-1/2})}{3(1-q^{-1/2})}<1.
\]
Since $\chi^2$ is nonprincipal of conductor degree $2m$, its explicit
formula and the modulus $q^{1/2}$ of its reciprocal zeros imply
$|\sum_{\deg P=d}\chi(P)^2|\ll_q mq^{d/2}/d$.
The contribution of square terms of degree $d>B$ is therefore
$O_q(mq^{-B/2}/(B+1))$. For $d\le B$, replacing
$e^{-2d/N}(1-2d/N)$ by one costs $O(B/N)$, using
$\#\{P:\deg P=d\}\le q^d/d$. Hence, for $N\ge2B$,
\begin{equation}\label{eq:GZsquare-shortening}
 \frac12\sum_{\deg P\le N/2}
 \frac{e^{-2\deg P/N}(1-2\deg P/N)\chi(P)^2}{|P|}
 =Q_B(\chi)+O_q\left(\frac BN+\frac{mq^{-B/2}}{B+1}\right).
\end{equation}
This error tends uniformly to zero for $N=N_j$.
Also $2m/N_j=\theta_j^{-1}+o(1)$ uniformly, since
$m\theta_0^2\to\infty$. Combining these facts proves the lemma.
\end{proof}

\subsection{Finite-model estimates and Taylor errors}

Expectations are taken with respect to \eqref{eq:model}. For
$Z=\sum c_PX_P$ and $V=\sum|c_P|^2$, the real and imaginary
exponential estimates in Lemma~\ref{lem:model-bounds} give
\[
 \Pr(|Z|>v)\le4e^{-v^2/(4V)},\qquad
 \E|Z|^{2s}\le4(4V)^ss!,\qquad
 \E e^{t|Z|}\le4e^{t^2V}.
\]
Integration of the tail estimate gives the moment bound.
The case $V=0$ is immediate.

\begin{lemma}\label{lem:GZTaylor}
Suppose $U=Z_u+Q$, $W=Z_w+Q$, where $Z_u,Z_w$ are prime sums
with coefficient square sums at most $V$, and $|Q|\le D$.
No independence between $Q$ and the prime sums is needed. If
$\ell\ge10^{10}(V+D^2+1)$, then
\begin{equation}\label{eq:GZTaylor}
 \|E_\ell(U/2)E_\ell(-W/2)-e^{(U-W)/2}\|_2\le e^{-\ell}.
\end{equation}
Moreover, if $|z|\le K$ and $\ell\ge100K$, then
\begin{equation}\label{eq:GZpointTaylor}
 E_\ell(\pm z/2)=e^{\pm z/2}(1+\eta),\qquad
 |\eta|\le e^{-2\ell}.
\end{equation}
\end{lemma}
\begin{proof}
Use $|E_\ell(z)|\le e^{|z|}$ and
$|E_\ell(z)-e^z|\le e^{|z|}|z|^{\ell+1}/(\ell+1)!$.
With $n=\ell+1$, the difference in \eqref{eq:GZTaylor} is at most
\[
 \frac{e^{(|U|+|W|)/2}}{n!}
 \big((|U|/2)^n+(|W|/2)^n\big).
\]
H\"older and the preceding exponential bounds give a factor at most
$\sqrt2e^{4V+D}$ for the exponential in its $L^2$ estimate.
Also $(\E|U|^{4n})^{1/(4n)}\le8\sqrt{nV}+D$, and the same
holds for $W$, by Lemma~\ref{lem:model-bounds} and Minkowski.
The norm is therefore at most
\[
 4e^{4V+2D}
 \left(8e\sqrt{\frac{V+D^2+1}{n}}\right)^n\le e^{-\ell}.
\]
The assumed lower bound on $\ell$ proves the last inequality.
For the pointwise statement, the relative error is at most
$e^K(eK/(2n))^n\le e^{-2\ell}$.
\end{proof}

Let $V_r=\sum_{\deg P\in I_r}|P|^{-1}$.
The prime polynomial bound gives $V_0\le1+\log(2m)$ and $V_r\le2$ for
$r\ge1$, uniformly for sufficiently large $m$. In the first block
$|Q_B(X)|\le D=(1+\log B)/2$. Consequently the hypothesis of
Lemma~\ref{lem:GZTaylor} holds for all blocks: for the first block
$\ell_0\gg(\log m)^{3/2}\gg V_0+D^2+1$, and for the later
blocks $\ell_r\ge100e^{153/4}>3\cdot10^{10}$. Furthermore,
\begin{equation}\label{eq:GZsmall-errors}
 \sum_{r=0}^J e^{-\ell_r}<10^{-3}.
\end{equation}
Indeed, with $a=100e^{153/4}$, we have
$\ell_r\ge ae^{3(J-r)/4}\ge a(1+3(J-r)/4)$, and hence
\[
 \sum_{r=0}^J e^{-\ell_r}
 \le\frac{e^{-a}}{1-e^{-3a/4}}<10^{-3}.
\]
This bound is uniform in $J$.

For $j\le J$, put
\[
 W_r(j)=\sum_{\deg P\in I_r}
 \frac{(a_{N_j}(P)-a_{N_J}(P))^2}{|P|}\quad(r\le j).
\]
Since $|a_{N_j}(P)-a_{N_J}(P)|\le2\deg P/N_j$,
\begin{equation}\label{eq:GZvariance}
 \sum_{r\le j}W_r(j)
 \le\frac4{N_j^2}\sum_{d\le N_j}d\le4.
\end{equation}
The common prime-square term cancels in $U-W$. Hence
\[
 \|e^{(U-W)/2}\|_2\le e^{W_r(j)/4},
\]
by Lemma~\ref{lem:model-bounds}. Applying
\eqref{eq:GZTaylor} to each block gives
\begin{equation}\label{eq:GZmatched}
 \prod_{r\le j}\E\left|
 E_{\ell_r}(\mathcal P_{r,N_j}(X)/2)
 E_{\ell_r}(-\mathcal P_{r,N_J}(X)/2)\right|^2
 \le e^3.
\end{equation}
Here the product of $(e^{W_r(j)/4}+e^{-\ell_r})^2$ is at most
$ e^{\frac12\sum W_r(j)+2\sum e^{-\ell_r}}<e^3$.

\subsection{The good set and the later exceptional blocks}

Let $\mathcal T_r$ be the set where
$|\mathcal P_{r,N_u}(\chi)|\le K_r$ for every $u\ge r$, and put
\[
 \mathcal T=\bigcap_{r=0}^J\mathcal T_r,\qquad
 \mathcal B_0=\mathcal T_0^c,\qquad
 \mathcal B_n=\left(\bigcap_{r<n}\mathcal T_r\right)\cap
 \mathcal T_n^c\quad(1\le n\le J).
\]
The sets $\mathcal T,\mathcal B_0,\ldots,\mathcal B_J$ partition the family. Equations
\eqref{eq:GZshort} and \eqref{eq:GZpointTaylor} imply
\begin{equation}\label{eq:GZgood}
 |\Lq(1/2,\chi)M_G(\chi)|\le e^{1/\theta_J+3}<e^{e^{52}+3}
 \qquad(\chi\in\mathcal T).
\end{equation}
The constant $3$ absorbs the $2$ in \eqref{eq:GZshort} and the
uniformly small product of Taylor errors.

\begin{proposition}\label{prop:GZlater}
Put $s_n=\lfloor1/(8\theta_n)\rfloor$ for $n\ge1$. Then
\begin{align}\label{eq:GZbadblock}
 q^{-2m}\sum_{\chi\in\mathcal B_n}|\Lq(1/2,\chi)M_G(\chi)|
 \le{}&e^{20}(J-n+1)e^{1/\theta_{n-1}+4(J-n)}
 \left(\frac{16s_n}{K_n^2}\right)^{s_n}.
\end{align}
Consequently
\begin{equation}\label{eq:GZbadlater}
 q^{-2m}\sum_{n=1}^J\sum_{\chi\in\mathcal B_n}|\Lq(1/2,\chi)M_G(\chi)|
 \le\frac{e^{20}}{(1-e^{-1})^2}e^{-0.03/\theta_J}.
\end{equation}
\end{proposition}
\begin{proof}
On $\mathcal B_n$ use the shorter scale $N_{n-1}$ in
\eqref{eq:GZshort}. For $r<n$ the prime blocks are controlled, so
\eqref{eq:GZpointTaylor} majorizes their exponentials by
$|E_{\ell_r}(\mathcal P_{r,N_{n-1}}/2)|^2$, with total correction
at most $e$. Insert
\[
 1_{\mathcal B_n}\le
 \sum_{u=n}^J|\mathcal P_{n,N_u}/K_n|^{2s_n}
\]
on the set where earlier blocks are controlled. The resulting
nonnegative finite polynomial is
\[
 \prod_{r<n}|E_{\ell_r}(\mathcal P_{r,N_{n-1}}/2)|^2
 \prod_{r=0}^J|E_{\ell_r}(-\mathcal P_{r,N_J}/2)|^2
 |\mathcal P_{n,N_u}/K_n|^{2s_n}.
\]
Its monomial radical degrees are bounded by
\begin{equation}\label{eq:GZdegreecheck}
 2L+2s_nN_n\le2L+m/2<0.510497m<m.
\end{equation}
Extend its sum to all monic $R\in\AQ$ of degree $m$ and apply
Lemma~\ref{lem:orthog}. The expectation then factors over blocks.
The product over $r<n$ is at most $e^3$ by
\eqref{eq:GZmatched}.

In block $n$, write $Z=\mathcal P_{n,N_u}(X)$ and
$W=\mathcal P_{n,N_J}(X)$. Both variances are at most two, and
H\"older gives
\begin{align*}
 \E\{|Z|^{2s}|E_{\ell_n}(-W/2)|^2\}
 &\le(\E|Z|^{4s})^{1/2}(\E e^{2|W|})^{1/2}\\
 &\le4e^4(16s)^s.
\end{align*}
Here $\sqrt{(2s)!}\le(2s)^s$.
Each later block contributes at most $\E e^{|W|}\le4e^2<e^4$.
The prefactors from \eqref{eq:GZshort} and
\eqref{eq:GZpointTaylor} are at most $e^{2+1/\theta_{n-1}}$ and
$e$, respectively. Multiplying these estimates and summing over $u$
gives \eqref{eq:GZbadblock}.

Now $1/(9\theta_n)\le s_n\le1/(8\theta_n)$, and
$16s_n/K_n^2\le2\sqrt{\theta_n}<1$.
Writing $t=J-n$ and using $\theta_{n-1}=\theta_n/e$, the logarithm
of the scale-dependent factors is at most
\[
 \frac{e+(\log2)/9-(51+t)/18}{\theta_n}+4t
 \le-\frac{0.03}{\theta_n}-t.
\]
Indeed,
\[
 \frac{51}{18}-e-\frac{\log2}{9}>0.038035>0.03,
\]
and $\theta_n\le e^{-51}$ ensures that $t/(18\theta_n)\ge5t$.
Since $\theta_n\le\theta_J$, summing
$(t+1)e^{-0.03/\theta_n-t}$ and using
$\sum_{t\ge0}(t+1)e^{-t}=(1-e^{-1})^{-2}$ gives
\eqref{eq:GZbadlater}.
\end{proof}

\subsection{The initial exceptional block}

\begin{lemma}\label{lem:GZrough}
There are fixed constants $C,C'$ such that
\[
 q^{-2m}\sumchi|\Lq(1/2,\chi)|^4\ll_q m^C,
 \qquad q^{-2m}\sumchi|M_G(\chi)|^4\ll_q m^{C'}.
\]
\end{lemma}
\begin{proof}
For the first bound, suppose
$256(1+\log(2m))\le V\le2m\log2$, put $N=\lfloor16m/V\rfloor$ and
$s=\lfloor V/64\rfloor$. Then $2sN\le m/2$,
$s\ge V/128$, and $2m/N\le V/4$. The prime-square absolute
contribution to \eqref{eq:log-upper} is at most $(1+\log(2m))/2$, and the higher
powers contribute less than one. Thus $\log|L|>V$ forces the
weighted prime sum $Z_N(\chi)$ to have real part greater than $V/2$.
Its variance is at most $1+\log(2m)$. Finite orthogonality and
Lemma~\ref{lem:model-bounds} imply, if $\mathcal N(V)$ counts
these characters,
\[
 q^{-2m}\mathcal N(V)\le
 \left(\frac{128s(1+\log(2m))}{V^2}\right)^s
 \le\left(\frac{2(1+\log(2m))}{V}\right)^{V/128}.
\]
For $V\ge2e^{2048}(1+\log(2m))$ this is at most $e^{-16V}$. Below that
threshold use $\mathcal N(V)\le q^{2m}$, and above
$2m\log2$ use $\mathcal N(V)=0$, since the factorization of the
completed $L$-polynomial gives $|L_q(1/2,\chi)|\le2^{2m}$.
Integration against $4e^{4V}\,dV$ proves a fixed polynomial bound.

For the second estimate, $4L<0.020994m<m$ permits the use of
Lemma~\ref{lem:orthog}. Since
$|E_\ell(-z/2)|^8\le e^{4|z|}$, its normalized average is at most
\[
 e^{4D}4^{J+1} e^{16\sum_{r=0}^JV_r}.
\]
Here $D=O(\log\log m)$, $J=O(\log\log m)$, and
$\sum V_r\le1+\log(2m)$. This is bounded by a fixed power of $m$.
\end{proof}

\begin{lemma}\label{lem:GZinitial}
One has
\begin{equation}\label{eq:GZinitial}
 \sum_{\chi\in\mathcal B_0}|\Lq(1/2,\chi)M_G(\chi)|=o(q^{2m}).
\end{equation}
\end{lemma}
\begin{proof}
Write $\mathcal P_{0,N_u}=Z_u+Q_B$.
Since $|Q_B|\le D=o(K_0)$, failure in the first block implies
$|Z_u|>K_0/2$ for some $u$. Take $s_0=\lfloor K_0\rfloor$.
The degree $2s_0N_0=o(m)$ permits exact finite orthogonality.
Markov and the model moment bound give
\[
 q^{-2m}|\mathcal B_0|\le
 4(J+1)\left(\frac{16(1+\log(2m))s_0}{K_0^2}\right)^{s_0}
 \le e^{-c(\log m)^{3/2}\log\log m}
\]
for some absolute $c>0$ and sufficiently large $m$.
H\"older with exponents $2,4,4$ bounds the left side of
\eqref{eq:GZinitial} by
\[
 |\mathcal B_0|^{1/2}
 \left(\sumchi|\Lq(1/2,\chi)|^4\right)^{1/4}
 \left(\sumchi|M_G(\chi)|^4\right)^{1/4}.
\]
Lemma~\ref{lem:GZrough} and the preceding bound for
$|\mathcal B_0|$ now give \eqref{eq:GZinitial}.
\end{proof}

\subsection{The mollified first moment}

For $\ell=4$, the arithmetic factor in \eqref{intro:arithmetic} is
\begin{equation}\label{eq:GZA}
 A_{4,q,G}=\prod_P\left(1+\rho_P\sum_{k\ge1}d_{4k}|P|^{-2k}\right).
\end{equation}
All $d_v$ are nonnegative and at most one: the exponent in
\eqref{eq:GZlocal} is coefficientwise bounded by
$\sum_{v\ge1}x^v/v$. Also $d_4=1/4$. Hence the product converges
absolutely, with
\begin{equation}\label{eq:GZAbounds}
 1<A_{4,q,G}\le\zeta_q(2),\qquad
 1+\rho_P\sum_{k\ge1}d_{4k}|P|^{-2k}
 =1+\frac{\rho_P}{4|P|^2}+O(|P|^{-4}).
\end{equation}

\begin{proposition}\label{prop:GZfirst}
The mollifier \eqref{eq:GZmollifier} satisfies
\begin{equation}\label{eq:GZfirst}
 \sumchi\Lq(1/2,\chi)M_G(\chi)
 =(c_qA_{4,q,G}+o(1))q^{2m}.
\end{equation}
\end{proposition}
\begin{proof}
Put $\lambda=0.005248338$, so that $L<\lambda m<m/10$ by
Lemma~\ref{lem:GZlength}. Apply the existing estimate
\eqref{eq:first-input} to each coefficient of \eqref{eq:GZmollifier}.
The total error is
\[
 \ll_{q,\varepsilon}q^{(9/5+\varepsilon)m}
 \sum_{\deg h\le L}|h|^{-1/2}|h^\ast|^{1/10+\varepsilon}|h|^\varepsilon
 \ll q^{(9/5+\varepsilon)m+(4/5+\varepsilon)L}=o(q^{2m}),
\]
after relabeling $\varepsilon$ and taking it sufficiently small, since
\[
 \frac95+\frac45\lambda<1.804199<2.
\]
The normalized main term is
\[
 c_q\sum_{h,D\in\Aq}
 \frac{b_m(h)\rho(hD)}{\sqrt{|hh^\ast|}|D|^2}.
\]
For each fixed $h$, all its prime factors eventually lie in the
first block and have degree at most $B$; their weights tend to one
and $\ell_0\to\infty$. Thus $b_m(h)$ tends to the multiplicative
coefficient whose prime series is $e^{-x-\frac{x^2}{2}}$.
For dominated convergence, use the coefficient
majorant $\prod_{P^v\Vert h}\alpha_v$ from
Lemma~\ref{lem:GZlength}. Its weighted sum is finite: the local
$h$-factor is
\[
 1+\sum_{v\ge1}\alpha_v|P|^{-2\lceil v/4\rceil}
 =1+O(|P|^{-2}),
\]
where $\sum_v\alpha_v=e^{3/2}$; the $D$-sum is $\zeta_q(2)$.
At each prime the diagonal
requires the total exponent from the mollifier and the $L$-series
to be divisible by four. Formula \eqref{eq:GZlocal} then gives
exactly the local factor in \eqref{eq:GZA}.
\end{proof}

\subsection{A lower bound for the non-vanishing proportion}

Let $N_4(g)=\#\{\chi\in\Cfour:\Lq(1/2,\chi)\ne0\}$.

\begin{theorem}\label{thm:GZpositive}
For the family \eqref{eq:family},
\begin{equation}\label{eq:GZpositive}
 \displaystyle
 \liminf_{\substack{g\to\infty\\3\mid g}}
 \frac{N_4(g)}{|\Cfour|}
 \ge\frac{A_{4,q,G}}{2 e^{e^{52}+3}}
 >\frac1{2 e^{e^{52}+3}}> e^{-e^{53}}.
\end{equation}
\end{theorem}
\begin{proof}
Proposition~\ref{prop:GZlater} and Lemma~\ref{lem:GZinitial} give
\[
 \sum_{\chi\notin\mathcal T}|\Lq(1/2,\chi)M_G(\chi)|
 \le\left(\frac{e^{20-0.03e^{51}}}{(1-e^{-1})^2}+o(1)\right)q^{2m}.
\]
For $q\ge3$, \eqref{eq:cq-lower} gives $c_q\ge8/27$.
Thus the constant on the right is less than $2/27\le c_qA_{4,q,G}/4$,
by \eqref{eq:GZAbounds}.
Subtracting this absolute contribution from \eqref{eq:GZfirst},
and taking $m$ sufficiently large, yields
\[
 \left|\sum_{\chi\in\mathcal T}\Lq(1/2,\chi)M_G(\chi)\right|
 \ge\tfrac12c_qA_{4,q,G}q^{2m}.
\]
Each nonzero summand on $\mathcal T$ has modulus at most
$e^{1/\theta_J+3}$ by \eqref{eq:GZgood}, and characters with
zero central value contribute zero. The triangle inequality gives
\[
 N_4(g)\ge\frac{c_qA_{4,q,G}}{2e^{1/\theta_J+3}}q^{2m}
 \ge\frac{c_qA_{4,q,G}}{2e^{e^{52}+3}}q^{2m}.
\]
Divide by $|\Cfour|\sim c_qq^{2m}$. Finally
$e^{52}+3+\log2<e^{53}$ proves the last inequality in
\eqref{eq:GZpositive}.
\end{proof}

\begin{remark}
The proof requires the uniform twist dependence in
Theorem~\ref{input:first} and the absolute exceptional-set estimates.
A sharp mollified second-moment bound is not needed.
\end{remark}

\subsection{Numerical forms}\label{subsec:positive}

Theorem~\ref{thm:GZpositive} gives, for each fixed $q\equiv3\pmod4$,
\[
 \liminf_{\substack{g\to\infty\\3\mid g}}
 \frac{N_4(g)}{|\Cfour|}
 \ge \frac{A_{4,q,G}}{2 e^{e^{52}+3}}
 > \frac1{2 e^{e^{52}+3}}.
\]
Here $A_{4,q,G}$ is the convergent Euler product \eqref{eq:GZA}, whose
first nonconstant local term is $\rho_P/(4|P|^2)$.
The family-counting constant $c_q$ cancels when the good-set first
moment is divided by the family size.

Moreover,
\[
 \log_{10}\bigl(2 e^{e^{52}+3}\bigr)
 =\frac{e^{52}+3+\log2}{\log10}<2\cdot10^{22}.
\]
Consequently the same lower limit is strictly greater than
\[
 \frac1{10^{\,2\cdot10^{22}}}.
\]
The bounds hold as $g\to\infty$ through multiples of three for the
family \eqref{eq:family}.

\section{Sextic case}\label{sec:sextic}

We treat the sextic family parameterized in \cite[Lemma~3.6]{Lin}.
A double Dirichlet series gives the required twisted first moment;
the truncation and exceptional-set estimates of Section~\ref{sec:GZ}
then give an explicit lower bound for the non-vanishing proportion.

\subsection{The family and its parameterization}

Throughout this section, $q\equiv5\pmod6$. We retain $Q=q^2$,
the notation $\Aq,\AQ$, and the norms of Section~\ref{sec:prelim};
thus $Q\equiv1\pmod{12}$.
Fix compatible identifications of the sixth roots of unity and define
\begin{equation}\label{sextic:family}
 \mathcal F_m=\{F\in\AQ:F\text{ square-free},\ \deg F=m,\ (F,F^\sigma)=1\},
 \qquad
 \chi_F=\left.\left(\frac{\cdot}{F}\right)_6\right|_{\F_q[T]}.
\end{equation}
Here $\sigma$ on polynomials denotes coefficient Frobenius, as before.
We write $\mathcal C_6(g)=\{\chi_F:F\in\mathcal F_m\}$ when
$g=5m-5$.

\begin{lemma}\label{sextic:parameterization}
The map $F\mapsto\chi_F$ is a bijection onto the primitive order-six
characters of conductor degree $2m$ satisfying
\begin{equation}\label{sextic:conductor}
 \cond(\chi^2)=\cond(\chi^3)=\cond(\chi).
\end{equation}
These characters are even, their conductor is $FF^\sigma$, and their
completed $L$-polynomial has degree $2m-2$.  The associated cyclic
sextic extension has genus $g=5m-5$.
\end{lemma}
\begin{proof}
A base-field prime supports a local character of exact order six
precisely when its degree is even.  It then splits as
$P=\pi\pi^\sigma$ over $\F_Q$, and the two choices $\pi,\pi^\sigma$
give its two local characters of exact order six.  Condition
\eqref{sextic:conductor} excludes local order two and local order three
at ramified primes.  Independent local choices give the parameterization,
including $(F,F^\sigma)=1$; see \cite[Lemma~3.6]{Lin}.
For $a\in\F_q^\times$, its sextic symbol is one since $q-1$ divides
$(Q^{\deg\pi}-1)/6$.  Thus the characters are even and infinity is
unramified.  Riemann--Hurwitz gives
$2g-2=-12+5(2m)$, whereas the degree of each completed $L$-polynomial
is its conductor degree minus two.
\end{proof}

The admissible conductor parameters are the same as in
Lemma~\ref{lem:count}.  Consequently
\begin{equation}\label{sextic:count}
 N_m:=|\mathcal F_m|=c_qq^{2m}+O_{q,\varepsilon}(q^{(1+\varepsilon)m}),
 \qquad c_q\ge(1-q^{-1})^3\ge64/125,
\end{equation}
where $c_q$ is \eqref{eq:cq}, evaluated at the present $q$.
The factors $\rho_P$ and $\rho(H)$ in \eqref{eq:rho} also remain
unchanged: imposing $(F,H)=1$ multiplies the leading counting
coefficient by $\rho(H)$.  Conjugate characters are counted separately.
Condition \eqref{sextic:conductor} is retained throughout;
the resulting proportion is relative to this family.

For a monic twist, write uniquely
\begin{equation}\label{sextic:hstar}
 h=C_1C_2^2C_3^3C_4^4C_5^5E^6,
 \qquad h^\ast=C_1^5C_2^4C_3^3C_4^2C_5,
\end{equation}
where $C_1,\ldots,C_5$ are square-free and pairwise coprime; no
coprimality with $E$ is imposed.  Then $hh^\ast$ is a sixth power,
$|h^\ast|\le|h|^5$, and
\[
 \chi_F(h)=\overline{\chi_F(h^\ast)}\qquad((F,h)=1).
\]
The coprimality restriction must be retained even when $h^\ast=1$.

\subsection{The double Dirichlet series and the first two regions}

Define
\begin{equation}\label{sextic:mds}
 A_6(s,w;h)=\sum_{d\ge1}\sum_{F\in\mathcal F_d}
       \frac{\chi_F(h)\Lq(w,\chi_F)}{|F|_2^s},
 \qquad B_{6,h}(s,w)=(1-Q^{1-s})A_6(s,w;h).
\end{equation}
Put $u=Q^{-s}=q^{-2s}$ and $v=q^{-w}$, writing
$A_6(u,v;h)$ for the same function in these variables.  Initially,
\begin{equation}\label{sextic:perron}
 \sum_{F\in\mathcal F_m}\chi_F(h)\Lq(1/2,\chi_F)
 =\frac1{2\pi i}\oint A_6(u,q^{-1/2};h)\frac{du}{u^{m+1}}.
\end{equation}

Let $K_6(u;k)$ be the character generating series with $F=1$ included.
For an even-degree prime $P=\pi\pi^\sigma$, put
$\alpha_P=\left(\frac{k}{\pi}\right)_6$, setting it to zero when $P\mid k$.  The
conjugate prime contributes $\overline{\alpha_P}$, so
\begin{equation}\label{sextic:character-product}
 K_6(u;k)=\prod_{\deg P\text{ even}}
 \bigl(1+(\alpha_P+\overline{\alpha_P})u^{\deg P/2}\bigr).
\end{equation}
By reciprocity,
\[
 K_6(Q^{-s};k)=L_Q(s,\chi^{(k)})\mathcal E_6(s;k),
 \qquad \chi^{(k)}(R)=\left(\frac{k}{R}\right)_6.
\]
The product $\mathcal E_6$ converges normally for $\Rea s>1/2$.
Indeed, removing the two local $L$-factors at an even prime cancels
the linear term and leaves $1+O(|P|^{-2\Rea s})$.  At an odd prime
not dividing $k$, the symbol on base-field arguments is trivial and
the correction is $1-|P|^{-2s}$; at primes dividing $k$ the factor is
one.  The character $\chi^{(k)}$ is principal exactly when $k$ is a
sixth power.  The fixed-field Lindel\"of bound, including finite
imprimitive factors, therefore gives
\begin{equation}\label{sextic:character-bound}
 (1-Q^{1-s})K_6(Q^{-s};k)\ll_{q,\eta}|k|^\eta
 \qquad(\Rea s>1/2).
\end{equation}
As throughout the quartic argument, bounds on an open region are
uniform on compact subsets of its real base and on all imaginary
translates.

Write $\sigma_0=\Rea s$ and $\tau=\Rea w$.  The character expansion
\begin{equation}\label{sextic:character-expansion}
 A_6(s,w;h)=\sum_{f\in\Aq}\frac{K_6(Q^{-s};hf)-1}{|f|^w}
\end{equation}
continues $B_{6,h}$ holomorphically to the tube over
\begin{equation}\label{sextic:S1}
 S_1=\{\sigma_0>1/2,\ \tau>1\},
 \qquad B_{6,h}(s,w)\ll_{q,\eta}|h|^\eta.
\end{equation}
The subtraction removes $F=1$ and is retained in every representation.
On the other hand, the functional equation \eqref{eq:FE} and Weil's
theorem give
\[
 |\Lq(w,\chi_F)|\ll_{q,\eta}
 |F|_2^{\max(0,1/2-\tau)+\eta}.
\]
Summing first over conductor degree proves normal convergence of the
original series in
\begin{equation}\label{sextic:S2}
 S_2=\{\sigma_0>1+\max(0,1/2-\tau)\}
     =\{\sigma_0>1,\ \sigma_0+\tau>3/2\}.
\end{equation}
The resulting bounds are independent of $h$.  The two representations
agree on their common domain of absolute convergence, hence throughout
their overlap.

\subsection{The sextic Gauss series bound and the non-Kummer restriction}

Use the standard compatible additive characters, with
$e_Q=e_q\circ\operatorname{Tr}_{\F_Q/\F_q}$, and write
\[
 G_6(r,F)=\sum_{a\bmod F}\left(\frac{a}{F}\right)_6e_Q(ar/F).
\]
For $F\in\mathcal F_d$, the root-number calculation of
Section~\ref{sec:first}, with sixth powers in place of fourth powers,
gives
\begin{equation}\label{sextic:root}
 \omega(\chi_F)=\xi_q^d q^{-d}G_6(1,F),
 \qquad \xi_q=(-1)^{(q+1)/6}.
\end{equation}
The Chinese remainder theorem identifies the base-field Gauss sum with
$\left(\frac{F^\sigma}{F}\right)_6G_6(1,F)$.  If $R=\operatorname{Res}(F,F^\sigma)$,
then $R^q=(-1)^dR$, and hence
$\left(\frac{F^\sigma}{F}\right)_6=R^{(Q-1)/6}=(-1)^{d(q+1)/6}$, with the roots of unity
identified as fixed above.  The remaining phase depends only on $d$ and is therefore absorbed
into the generating variable.

For $0\le i\le5$, let
\begin{equation}\label{sextic:gauss-definitions}
 \psi^{(i)}(r,z)=\frac1{1-Q^6z^6}
   \sum_{\deg K\equiv i\ (6)}G_6(r,K)z^{\deg K},
 \qquad
 \widetilde\Psi_6(r,z;h)=
   \sum_{\substack{K\in\AQ\\(K,rh)=1}}G_6(r,K)z^{\deg K}.
\end{equation}
Theorem~4.1 of \cite{DFL25}, over $\F_Q$, gives
\[
 \psi^{(i)}(r,z)\ll_{q,\eta}|r|_2^{3/4-\beta/2+\eta},
 \qquad |z|=Q^{-\beta},\quad 4/3<\beta<3/2.
\]
The possible global poles $z^6=Q^{-7}$ have $\beta=7/6$ and do not
meet this range.  We require the following restricted form.

\begin{lemma}\label{sextic:restricted-gauss}
The function $\widetilde\Psi_6(r,z;h)$ is holomorphic for
$|z|<Q^{-7/6}$.  Uniformly on closed substrips of $4/3<\beta<3/2$,
\begin{equation}\label{sextic:restricted-bound}
 \widetilde\Psi_6(r,z;h)\ll_{q,\eta}
 |rh|_2^\eta |r|_2^{3/4-\beta/2}
 \qquad(|z|=Q^{-\beta}).
\end{equation}
For $\beta>3/2$, the defining series converges absolutely and is
bounded independently of $r,h$ on closed substrips.
\end{lemma}
\begin{proof}
Let $a$ be the product of the extension-field primes dividing $h$
but not $r$.  In the notation of \cite[Section~5]{DFL25}, finite
inclusion--exclusion gives
\begin{equation}\label{sextic:restricted-inclusion}
 \widetilde\Psi_6(r,z;h)
 =(1-Q^6z^6)\sum_{E\mid a}\mu_Q(E)
                 \sum_{i=0}^5\psi_r^{(i)}(r,E,z),
\end{equation}
where $\psi_r^{(i)}(r,E,z)$ imposes $(K,r)=1$ and $E\mid K$.
Write $r=r_1r_2^2\cdots r_5^5r_6^6$ as required in
\cite[Theorem~5.3 and Corollary~5.4]{DFL25}.  Their second case applies:
$6\le6/(3-2\beta)$ throughout the stated range.  The bound is
\[
 \psi_r^{(i)}(r,E,z)
 \ll_{q,\eta}|Er|_2^\eta
 |E|_2^{7/2-3\beta}|r|_2^{3/4-\beta/2}.
\]
Since $7/2-3\beta<-1/2$, summing over $E\mid a$ costs only
$|h|_2^\eta$, after relabeling small exponents.  This proves
\eqref{sextic:restricted-bound}, retaining primes in the sixth-power
part of the original twist.  The finite decomposition in the cited
theorem has possible denominators $1-Q^7z^6$ and
$1-Q^{5\deg\pi}z^{6\deg\pi}$; the reciprocity signs are one since
$Q\equiv1\pmod{12}$.  Their zeros have radii $Q^{-7/6}$ and
$Q^{-5/6}$, respectively, proving the stated holomorphy.  Finally, on
the coprime support the Gauss sum vanishes for nonsquare-free $K$ and
otherwise has size $|K|_2^{1/2}$.  Thus absolute convergence follows
from the majorant $\zeta_Q(\beta-1/2)$ for $\beta>3/2$.
\end{proof}

For the dual representation put $r=fh^\ast$ and $\beta=\sigma_0+\tau$.
Initially $\tau<0$ and $\beta>3/2$.  The functional equation and
\eqref{sextic:root}, followed by the exact sieve
\[
 \mathbf1_{(F,F^\sigma)=1}
 =\sum_{\substack{D\in\Aq\\D\mid F}}\mu_q(D)
 \qquad(F\text{ square-free}),
\]
give
\begin{align}\label{sextic:dual}
 A_6(s,w;h)
 ={}&q^{2w-1}\frac{1-q^{-w}}{1-q^{w-1}}\Bigg\{
 \sum_{f\in\Aq}|f|^{w-1}
 \sum_{\substack{D\in\Aq\\(D,fh)=1}}
 \frac{\mu_q(D)\xi_q^{\deg D}G_6(r,D)}{|D|_2^{s+w}}\notag\\
 &\hspace{21mm}\times\widetilde\Psi_6(rD^4,\xi_qQ^{-(s+w)};h)
 -\zeta_q(1-w)\Bigg\}.
\end{align}
Indeed, Gauss multiplicativity and sextic reciprocity give
$G_6(r,DK)=G_6(r,D)G_6(rD^4,K)$ when $(D,K)=1$.
The vanishing of nonsquare-free coprime Gauss sums removes that
restriction on $K$.  The term $\zeta_q(1-w)$ subtracts the empty
conductor $D=K=1$.  Thus the non-Kummer sieve produces the twist
$rD^4$, rather than the quartic twist $rD^2$.

\subsection{The dual domain and analytic continuation}

For $4/3<\beta<3/2$, Lemma~\ref{sextic:restricted-gauss} bounds the
inner series in \eqref{sextic:dual} by
\[
 |h|^\eta |fh^\ast D^4|_2^{3/4-\beta/2+\eta}
 \ll |h|^{O(\eta)}|fh^\ast|^{3/2-\beta+O(\eta)}
                    |D|^{6-4\beta+O(\eta)}.
\]
Together with $|G_6(r,D)|\le |D|$, the resulting majorant is
\begin{equation}\label{sextic:majorant}
 |h|^\eta|h^\ast|^{3/2-\beta+\eta}
 \sum_{f\in\Aq}|f|^{1/2-\sigma_0+\eta}
 \sum_{D\in\Aq}|D|^{7-6\beta+\eta}+O(1),
\end{equation}
after relabeling $\eta$.  The degree-$d$ contribution of the last sum is
$q^{(8-6\beta+\eta)d}$.  Normal convergence therefore requires
$\sigma_0>3/2$ and $\beta>4/3$.  For $\beta>3/2$, absolute convergence of the Gauss series gives a
bound independent of $r,h$.  A neighborhood of $\beta=3/2$ is obtained
by applying the maximum-modulus principle on a slightly larger circle
$|z|=Q^{-3/2+\nu}$ in Lemma~\ref{sextic:restricted-gauss}, with
$\nu>0$ chosen within the strict margins $\sigma_0>3/2$ and $\tau<0$.
Thus $B_{6,h}$ is holomorphic on the tube over
\begin{equation}\label{sextic:S3}
 S_3=\{\sigma_0>3/2,\ \sigma_0+\tau>4/3,\ \tau<0\}.
\end{equation}
For $4/3<\beta<3/2$ its bound is
\begin{equation}\label{sextic:S3bound}
 B_{6,h}(s,w)\ll_{q,\eta}
 |h|^\eta|h^\ast|^{3/2-\sigma_0-\tau+\eta}.
\end{equation}

The character and family expansions agree in their common region of
absolute convergence. The dual expansion agrees with them where
$\tau<0$ and $\sigma_0+\tau>3/2$. The identity theorem on the
connected overlaps gives a single holomorphic function on
$T(S_1\cup S_2\cup S_3)$.
Bochner's tube theorem, as used in Section~\ref{sec:first}, extends it
to the tube over
\begin{equation}\label{sextic:hull}
 \operatorname{conv}(S_1\cup S_2\cup S_3)
 =\left\{\sigma_0>\frac12,\quad
 \sigma_0+\tau>\frac43,\quad
 7\sigma_0+6\tau>\frac{19}{2}\right\}.
\end{equation}
In $(u,v)$ coordinates,
\eqref{sextic:hull} becomes
\begin{equation}\label{sextic:hull-uv}
 |u|<q^{-1},\qquad |u|^{1/2}|v|<q^{-4/3},\qquad
 |u|^7|v|^{12}<q^{-19}.
\end{equation}
At $w=1/2$, its slice is $\sigma_0>13/14$.
The global Gauss poles have $\beta=7/6$, and the local denominators
have $\beta=5/6$; both lie outside the region used.  Hence the only
possible poles of $A_6$ inside the hull come from $(1-Q^{1-s})^{-1}$.

For uniformity in $h$, interpolate the bounds on $S_1$ and $S_3$
as in Section~\ref{sec:first}. At real points
$(1/2+\delta,1+\delta)$ and $(3/2+\delta,-1/6+\delta)$, they are
$O(|h|^\eta)$ and
$O(|h|^\eta|h^\ast|^{1/6-2\delta+\eta})$, respectively.
The three-lines theorem gives a twist exponent at most $1/14+\eta$
on the central line $\Rea s=13/14+13\delta/7$.
The periods $\pi/\log q$ and $2\pi/\log q$ ensure boundedness
on the interpolation strips. Taking $\delta=7\varepsilon/26$,
choosing $\eta$ sufficiently small, and using $|h^\ast|\le|h|^5$
yields
\begin{equation}\label{sextic:contour-bound}
 A_6(s,1/2;h)\ll_{q,\varepsilon}
 |h^\ast|^{1/14}|h|^\varepsilon,
 \qquad \Rea s=\frac{13}{14}+\frac\varepsilon2.
\end{equation}
This estimate is uniform in $h$ and $\Im s$.

\subsection{The counting pole and the twisted first moment}

In \eqref{sextic:character-expansion}, the counting pole at $s=1$
occurs precisely when $hf$ is a sixth power.  Writing $f=h^\ast A^6$
and using \eqref{eq:rho}, we obtain initially for $\tau>1$
\begin{equation}\label{sextic:residue}
 B_{6,h}(1,w)=c_q|h^\ast|^{-w}
                \sum_{A\in\Aq}\rho(hA)|A|^{-6w}.
\end{equation}
The right side is holomorphic for $\tau>1/6$, while the $s=1$ slice
of \eqref{sextic:hull} contains $\tau>5/12$.  The identity theorem
thus continues this equality to $w=1/2$.  The density is $\rho(hA)$,
not $\rho(h^\ast A)$: primes in the sixth-power part of $h$ still
exclude conductors.

\begin{theorem}[Sextic twisted first moment]\label{sextic:tfm}
For fixed $q\equiv5\pmod6$ and every $\varepsilon>0$, uniformly in
monic $h$ with $\deg h\le m/10$, one has
\begin{align}\label{sextic:tfm-formula}
 \sum_{F\in\mathcal F_m}\chi_F(h)\Lq(1/2,\chi_F)
 ={}&q^{2m}\mathcal M_{6,q}(h)\notag\\
 &+O_{q,\varepsilon}\!\left(q^{(13/7+\varepsilon)m}
                     |h^\ast|^{1/14}|h|^\varepsilon\right),
\end{align}
where
\begin{align}\label{sextic:main}
 \mathcal M_{6,q}(h)
 ={}&\frac{c_q}{\sqrt{|h^\ast|}}
                  \sum_{A\in\Aq}\frac{\rho(hA)}{|A|^3}\notag\\
 ={}&\frac{c_q\zeta_q(3)}{\sqrt{|h^\ast|}}
  \prod_P\bigl(1-(1-\rho_P)|P|^{-3}\bigr)
  \prod_{P\mid h}\frac{\rho_P}{1-(1-\rho_P)|P|^{-3}}.
\end{align}
\end{theorem}
\begin{proof}
Periodicity and \eqref{sextic:hull} continue
$A_6(u,q^{-1/2};h)$ meromorphically to $|u|<q^{-13/7}$, with the
only pole at $u=Q^{-1}$.  Its principal part is
$\mathcal M_{6,q}(h)/(1-Qu)$ by \eqref{sextic:residue}.
For $0<\varepsilon<1/7$, move the contour in
\eqref{sextic:perron} to $|u|=q^{-13/7-\varepsilon}$.  The counting
pole contributes $q^{2m}\mathcal M_{6,q}(h)$; explicitly, the residue
of the Perron integrand there is $-q^{2m}\mathcal M_{6,q}(h)$.
The remaining integral is bounded using \eqref{sextic:contour-bound},
which proves the error term.  The Euler product follows from the
geometric series at each prime.
\end{proof}

\begin{corollary}\label{sextic:first-moment}
Fix an odd prime power $q\equiv5\pmod6$. For every $\varepsilon>0$,
as $m\to\infty$,
\begin{equation}\label{sextic:first-moment-formula}
 \sum_{F\in\mathcal F_m}\Lq(1/2,\chi_F)
 =q^{2m}\mathcal M_{6,q}(1)
  +O_{q,\varepsilon}\!\left(q^{(13/7+\varepsilon)m}\right),
\end{equation}
where
\begin{equation}\label{sextic:first-moment-constant}
 \mathcal M_{6,q}(1)
 =c_q\zeta_q(3)
   \prod_P\left(1-(1-\rho_P)|P|^{-3}\right)>0.
\end{equation}
\end{corollary}
\begin{proof}
Set $h=1$ in Theorem~\ref{sextic:tfm}. Then $h^\ast=1$, and
\eqref{sextic:main} gives the stated constant. Its Euler product
converges absolutely and has positive local factors.
\end{proof}

\begin{remark}\label{sextic:twist-length}
Theorem~\ref{sextic:tfm} gives a uniform absolute-error formula; a
relative asymptotic is not asserted for every growing twist in its range.  If $|b(h)|\le1$ and
$\deg h\le L\le m/10$, its summed error is at most
\begin{equation}\label{sextic:summed-error}
 q^{(13/7+\varepsilon)m}
 \sum_{\deg h\le L}|h|^{-1/2}|h^\ast|^{1/14}|h|^\varepsilon
 \ll q^{(13/7+\varepsilon)m+(6/7+\varepsilon)L}.
\end{equation}
The length in \eqref{eq:GZlength} therefore leaves a fixed power
saving relative to $q^{2m}$.
\end{remark}

\subsection{The squared exponential mollifier and its first moment}

Use the parameters \eqref{eq:GZparameters} for the
sextic family, so $e^{-52}<\theta_J\le e^{-51}$ and
$B=\lceil4\log_q(2m)\rceil$.  Put
\begin{align}\label{sextic:blocks}
 Q_B(\chi)&=\frac12\sum_{\deg P\le B}\frac{\chi(P)^2}{|P|},\notag\\
 \mathcal P^{(6)}_{r,N_u}(\chi)
 &=\sum_{\deg P\in I_r}\frac{a_{N_u}(P)\chi(P)}{\sqrt{|P|}}
                       +1_{\{r=0\}}Q_B(\chi),\notag\\
 M_{6,G}(\chi)&=\prod_{r=0}^J
     E_{\ell_r}\!\left(-\frac{\mathcal P^{(6)}_{r,N_J}(\chi)}2\right)^2.
\end{align}
The prime-square term is common to all scales, as in
Section~\ref{sec:GZ}.
Lemma~\ref{lem:GZlength} is a coefficient calculation independent of
the character order; it yields
\begin{equation}\label{sextic:length}
 M_{6,G}(\chi)=\sum_{\deg h\le L}\frac{b_m(h)\chi(h)}{\sqrt{|h|}},
 \qquad |b_m(h)|\le1,\qquad L/m<0.005248338.
\end{equation}

Retain the nonnegative coefficients $d_v$ from \eqref{eq:GZlocal},
namely
\begin{equation}\label{sextic:local-series}
 \frac{e^{-x-\frac{x^2}{2}}}{1-x}
 = e^{\sum_{v\ge3}\frac{x^v}{v}}
 =\sum_{v\ge0}d_vx^v,
 \qquad d_6=\frac16+\frac1{18}=\frac29.
\end{equation}
The sextic arithmetic factor is
\begin{equation}\label{sextic:arithmetic}
 A_{6,q,G}=\prod_P\left(1+\rho_P\sum_{k\ge1}d_{6k}|P|^{-3k}\right).
\end{equation}
Since $0\le d_v\le1$, it converges absolutely and satisfies
\begin{equation}\label{sextic:arithmetic-bounds}
 1<A_{6,q,G}\le\zeta_q(3),\qquad
 1+\rho_P\sum_{k\ge1}d_{6k}|P|^{-3k}
 =1+\frac{2\rho_P}{9|P|^3}+O(|P|^{-6}).
\end{equation}

\begin{proposition}\label{sextic:mollified-first}
For the mollifier \eqref{sextic:blocks},
\begin{equation}\label{sextic:first}
 \sum_{F\in\mathcal F_m}\Lq(1/2,\chi_F)M_{6,G}(\chi_F)
 =(c_qA_{6,q,G}+o(1))q^{2m}.
\end{equation}
\end{proposition}
\begin{proof}
Put $\lambda=0.005248338$. By \eqref{sextic:length}, every twist
has degree at most $L<\lambda m<m/10$, so
Theorem~\ref{sextic:tfm} applies. Insert \eqref{sextic:tfm-formula}
into \eqref{sextic:length}. By \eqref{sextic:summed-error}, the total
error is
\[
 O_{q,\varepsilon}\!\left(
 q^{(13/7+\varepsilon)m+(6/7+\varepsilon)L}\right)=o(q^{2m}),
 \qquad \frac{13}{7}+\frac67\lambda<1.861642<2,
\]
with $\varepsilon$ sufficiently small.
For fixed $h$, its primes eventually lie in the first block below $B$,
the weights tend to one, and the cutoffs tend to infinity.  The
coefficient limit therefore has local series $e^{-x-\frac{x^2}{2}}$.
The absolute coefficient majorant from Lemma~\ref{lem:GZlength} has
local coefficients $\alpha_v$ with $\sum_v\alpha_v=e^{3/2}$.  In the
main term its local $h$-sum is
\[
 1+\sum_{v\ge1}\alpha_v|P|^{-3\lceil v/6\rceil}
 =1+O(|P|^{-3}),
\]
and the remaining $A$-sum is bounded by $\zeta_q(3)$.  Dominated
convergence applies.  The total exponent of $fh$ must be divisible
by six; selecting those coefficients in \eqref{sextic:local-series}
gives precisely \eqref{sextic:arithmetic}.
\end{proof}

\subsection{The good set and the absolute exceptional first moment}

We record why the estimates from Section~\ref{sec:GZ} transfer to
order six, rather than assuming a sextic twisted second moment.
For arbitrary monic $R\in\AQ$, define the restricted sextic symbol
$\chi_R$.  For monic $U,V\in\Aq$ with $\deg\rad(UV)\le m$,
sextic reciprocity and complete residue-class averaging give
\begin{equation}\label{sextic:orthogonality}
 q^{-2m}\sum_{\substack{R\in\AQ\\\deg R=m}}
        \chi_R(U)\overline{\chi_R(V)}
 =\begin{cases}
 \displaystyle\frac{\phi_Q(\rad(UV))}{|\rad(UV)|_2},
                 &UV^5\text{ is a sixth power},\\[2mm]
 0,              &\text{otherwise}.
 \end{cases}
\end{equation}
The proof is that of Lemma~\ref{lem:orthog}, with $UV^5$ in place of
$UV^3$. The character in the variable $R$ is induced modulo
$\rad(UV)$, which is the modulus used in the complete average.
Thus the exact independent model now has
\[
 \Pr(X_P=0)=1-\gamma_P,\qquad
 \Pr(X_P=\xi)=\frac{\gamma_P}{6}\quad(\xi^6=1),
\]
with the same $\gamma_P$ as in \eqref{eq:gamma}.  The mean is zero and
$|X_P|\le1$, so all prime-sum bounds in
Lemma~\ref{lem:model-bounds} remain valid with the same constants.
Only finite polynomials satisfying the monomial degree condition are
transferred to this model.

The sextic square $\chi^2$ is a nonprincipal primitive cubic
character of conductor degree $2m$.  Its explicit formula therefore
gives
\[
 \left|\sum_{\deg P=d}\chi(P)^2\right|\ll_q\frac{m q^{d/2}}d.
\]
Consequently the prime-square shortening
\eqref{eq:GZsquare-shortening} holds without change.  The logarithmic
upper bound derived from \cite[Lemma~3.1]{DFL21} before separating
prime powers uses only Weil's theorem and conductor degree, and yields
\begin{equation}\label{sextic:log-bound}
 \log|\Lq(1/2,\chi)|
 \le\Rea\sum_{r=0}^j\mathcal P^{(6)}_{r,N_j}(\chi)
                          +\theta_j^{-1}+2.
\end{equation}
Higher prime powers are absolutely summable, and the error in
shortening the square sum is uniform and tends to zero.  Here the explicit formula is applied to the cubic character $\chi^2$.

Define $\mathcal T^{(6)}$ and $\mathcal B_n^{(6)}$ by the thresholds
$|\mathcal P^{(6)}_{r,N_u}|\le K_r$ exactly as in
Section~\ref{sec:GZ}.  The Taylor estimate in
Lemma~\ref{lem:GZTaylor} uses only the preceding prime-sum bounds and
$|Q_B(X)|\le(1+\log B)/2$, not the order of $X_P$.  Moreover the
common square term cancels between the two scales in its difference
$U-W$.  Equations \eqref{eq:GZvariance} and \eqref{eq:GZmatched}
therefore retain their constants. The hypotheses of the Taylor
estimates hold for the parameters \eqref{eq:GZparameters}, as
verified in Section~\ref{sec:GZ}. It follows that
\begin{equation}\label{sextic:good}
 |\Lq(1/2,\chi)M_{6,G}(\chi)|
 \le e^{\theta_J^{-1}+3}<e^{e^{52}+3}
 \qquad(\chi\in\mathcal T^{(6)}).
\end{equation}

The later exceptional blocks satisfy the same estimate.  With $s_n=\lfloor1/(8\theta_n)\rfloor$ and $n\ge1$,
use the shorter scale $N_{n-1}$ in \eqref{sextic:log-bound} and insert
\[
 1_{\mathcal B_n^{(6)}}
 \le\sum_{u=n}^J
       |\mathcal P^{(6)}_{n,N_u}/K_n|^{2s_n}
\]
on the set where earlier blocks are controlled.  The majorant is a
nonnegative finite polynomial whose monomial radical degrees are
bounded by
\[
 2L+2s_nN_n\le2L+\frac m2<0.510497m<m.
\]
Thus \eqref{sextic:orthogonality} applies after extending its sum to
all $R$. The matched earlier blocks cost
at most $e^3$.  The failing block costs at most
$4e^4(16s_n)^{s_n}/K_n^{2s_n}$, by H\"older and the same moment bound
as in Proposition~\ref{prop:GZlater}; each later block costs at most
$e^4$.  Thus
\begin{align}\label{sextic:bad-block}
 q^{-2m}\sum_{\chi\in\mathcal B_n^{(6)}}
    |\Lq(1/2,\chi)M_{6,G}(\chi)|
 \le{}&e^{20}(J-n+1)e^{1/\theta_{n-1}+4(J-n)}
               \left(\frac{16s_n}{K_n^2}\right)^{s_n},\\
 q^{-2m}\sum_{n=1}^J\sum_{\chi\in\mathcal B_n^{(6)}}
    |\Lq(1/2,\chi)M_{6,G}(\chi)|
 \le{}&\frac{e^{20}}{(1-e^{-1})^2}e^{-0.03/\theta_J}.
 \label{sextic:bad-later}
\end{align}
The summation uses $16s_n/K_n^2\le2\sqrt{\theta_n}$ and
$\theta_J\le e^{-51}$. In particular,
\[
 \frac{51}{18}-e-\frac{\log2}{9}>0.038035>0.03,
\]
so the summation argument in Proposition~\ref{prop:GZlater} applies
with the same constant $0.03$.

Finally the initial block satisfies
\begin{equation}\label{sextic:bad-initial}
 q^{-2m}|\mathcal B_0^{(6)}|
 \le e^{-c(\log m)^{3/2}\log\log m},
 \qquad
 \sum_{\chi\in\mathcal B_0^{(6)}}
       |\Lq(1/2,\chi)M_{6,G}(\chi)|=o(q^{2m}).
\end{equation}
To justify the second assertion, the proof of Lemma~\ref{lem:GZrough}
uses only \eqref{sextic:log-bound}, or its unshortened prime-power
form, and \eqref{sextic:orthogonality}.  It gives fixed polynomial
bounds for $q^{-2m}\sum|L|^4$ and $q^{-2m}\sum|M_{6,G}|^4$.
For the latter average, the required radical-degree bound follows
from $4L<0.020994m<m$, by \eqref{sextic:length}.
For the first assertion take $s_0=\lfloor K_0\rfloor$;
$2s_0N_0=o(m)$ permits exact averaging, while
$|Q_B|=o(K_0)$.  Markov's inequality gives the displayed decay.
H\"older with exponents $2,4,4$ then proves the absolute weighted
bound, as in Lemma~\ref{lem:GZinitial}.  These auxiliary rough
moments are proved from finite averages; no asymptotic second moment
is being assumed.

\subsection{A lower bound for the non-vanishing proportion}

Let
\[
 N_6(m)=\#\{F\in\mathcal F_m:\Lq(1/2,\chi_F)\ne0\}.
\]

\begin{theorem}\label{sextic:positive}
For every fixed $q\equiv5\pmod6$, the family
\eqref{sextic:family} satisfies
\begin{equation}\label{sextic:proportion}
 \liminf_{m\to\infty}\frac{N_6(m)}{N_m}
 \ge\frac{A_{6,q,G}}{2 e^{e^{52}+3}}
 >\frac1{2 e^{e^{52}+3}}
 > e^{-e^{53}}.
\end{equation}
Equivalently, this holds for $\mathcal C_6(g)$ as $g\to\infty$
through multiples of five.
\end{theorem}
\begin{proof}
By \eqref{sextic:bad-later} and \eqref{sextic:bad-initial},
\[
 \sum_{\chi\notin\mathcal T^{(6)}}
        |\Lq(1/2,\chi)M_{6,G}(\chi)|
 \le\left(\frac{e^{20-0.03e^{51}}}{(1-e^{-1})^2}+o(1)\right)q^{2m}.
\]
The explicit constant is less than $c_qA_{6,q,G}/4$, since
$c_q\ge64/125$ and $A_{6,q,G}>1$.
Subtract this absolute contribution from \eqref{sextic:first}.
For sufficiently large $m$,
\[
 \left|\sum_{\chi\in\mathcal T^{(6)}}
        \Lq(1/2,\chi)M_{6,G}(\chi)\right|
 \ge\tfrac12c_qA_{6,q,G}q^{2m}.
\]
A zero central value contributes zero.  Every other summand on the
good set has modulus at most $ e^{\theta_J^{-1}+3}$ by
\eqref{sextic:good}.  The triangle inequality gives
\[
 N_6(m)\ge\frac{c_qA_{6,q,G}}{2 e^{\theta_J^{-1}+3}}q^{2m}
          \ge\frac{c_qA_{6,q,G}}{2 e^{e^{52}+3}}q^{2m}.
\]
Divide by \eqref{sextic:count}. The strict inequalities follow from
$A_{6,q,G}>1$ and $e^{52}+3+\log2<e^{53}$.
\end{proof}

The sextic constant uses the same truncation parameters and good-set
bound as the quartic theorem.  Its arithmetic factor is
$A_{6,q,G}$, with first nonconstant local term
$2\rho_P/(9|P|^3)$, in place of the quartic term
$\rho_P/(4|P|^2)$.  In particular, the family-counting factor $c_q$
cancels in the proportion.  The factor is determined by the local series
$e^{-x-\frac{x^2}{2}}/(1-x)$.

\section{Cubic case}\label{sec:cubic}

We apply the truncation argument of Section~\ref{sec:GZ} to the full
non-Kummer cubic family, using the twisted first moment, short-prime
inequality, and fourth-moment bound of David, Florea, and
Lal\'in~\cite{DFL21}. The common truncation parameters and
the exceptional-set estimate give a proportion exceeding
$e^{-e^{53}}$. The twisted input throughout this section is the
published Proposition~8.1 of \cite{DFL21}.

\subsection{The full family and finite orthogonality}

Throughout this section, $q$ is an odd prime power with
$q\equiv2\pmod3$. Put $Q=q^2$ and $m=g/2+1$, where $g$ is even.
The full primitive cubic family is
\begin{equation}\label{cubic:family}
 \mathcal C_3(g)=\left\{\chi_F=
 \left.\left(\frac{\cdot}{F}\right)_3\right|_{\F_q[T]}:
 F\in\AQ\text{ square-free},\ \deg F=m,\ (F,F^\sigma)=1\right\}.
\end{equation}
This is the parameterization in \cite[Section~2]{DFL21}.
Every character is even, has conductor $FF^\sigma$ of degree $2m$,
and has a completed $L$-polynomial of degree $g=2m-2$.
The admissible polynomials are the same as in Lemma~\ref{lem:count},
so
\begin{equation}\label{cubic:count}
 |\mathcal C_3(g)|=c_qq^{2m}+O_{q,\varepsilon}(q^{(1+\varepsilon)m}),
 \qquad c_q\ge(1-q^{-1})^3\ge\frac{64}{125}.
\end{equation}
The factors $\rho_P$ in \eqref{eq:rho} are unchanged.

For arbitrary monic $R\in\AQ$, write
$\chi_R(f)=\left(\frac{f}{R}\right)_3$ on $\Aq$.
The same reciprocity and completion argument as in the proof of
\cite[Lemma~3.6]{DFL21} gives the following form of finite orthogonality.

\begin{lemma}\label{cubic:orthog}
If $f,h\in\Aq$ and $\deg\operatorname{rad}(fh)\le m$, then
\begin{equation}\label{cubic:orthog-formula}
 q^{-2m}\sum_{\substack{R\in\AQ\\\deg R=m}}
 \chi_R(f)\overline{\chi_R(h)}
 =\begin{cases}
 \displaystyle\frac{\phi_Q(fh)}{|fh|_2},&fh^2\text{ is a cube},\\[2mm]
 0,&\text{otherwise}.
 \end{cases}
\end{equation}
In particular, the identity applies term by term to
$|H(\chi)|^2$, with $H(\chi)=\sum_{\deg f\le D}a(f)\chi(f)$,
whenever $2D\le m$.
\end{lemma}
\begin{proof}
Cubic reciprocity gives
$\chi_R(f)\overline{\chi_R(h)}=\left(\frac{R}{fh^2}\right)_3$.
The character on the right is induced modulo $\operatorname{rad}(fh)$,
with value zero on nonunits. Monic polynomials of degree $m$ are
uniformly distributed modulo this radical. The complete sum vanishes
unless $fh^2$ is a cube; in the principal case it is the proportion
of invertible residue classes, namely $\phi_Q(fh)/|fh|_2$.
This argument applies to arbitrary monic $R$, without a square-free
restriction. Finally, $\deg\operatorname{rad}(fh)\le2D$ for the
monomials arising from $|H|^2$.
\end{proof}

Let $X_P$ be independent variables with
\begin{equation}\label{cubic:model}
 \Pr(X_P=0)=1-\gamma_P,\qquad
 \Pr(X_P=\xi)=\frac{\gamma_P}{3}\quad(\xi^3=1),
\end{equation}
where $\gamma_P$ is defined in \eqref{eq:gamma}. Lemma~\ref{cubic:orthog}
gives
\begin{equation}\label{cubic:finite-average}
 \sum_{\chi\in\mathcal C_3(g)}|H(\chi)|^2
 \le q^{2m}\E|H(X)|^2\qquad(2D\le m).
\end{equation}
The expectation represents the completed finite character average;
the original restricted family is not assumed to have independent
local values. In particular, $\gamma_P$ and $\rho_P$ have different
roles.

\subsection{The analytic inputs and the mollifier}

By \cite[Lemma~3.1 and Proposition~6.1]{DFL21},
\begin{align}
 \log|L_q(1/2,\chi)|
 &\le\Rea\sum_{\deg f\le N}
 \frac{\Lambda(f)\chi(f)(N-\deg f)}
      {N|f|^{1/2+1/(N\log q)}\deg f}
 +\frac{2m}{N},\qquad 1\le N\le2m,
 \label{cubic:short-input}\\
 \sum_{\chi\in\mathcal C_3(g)}|L_q(1/2,\chi)|^4
 &\ll_{q,\varepsilon}q^{2m}m^{4+\varepsilon}.
 \label{cubic:fourth}
\end{align}
The convention in the first inequality is $\log0=-\infty$.

Use the parameters in \eqref{eq:GZparameters}, so that
$e^{-52}<\theta_J\le e^{-51}$ and $J=O(\log\log m)$.
The cutoff $B$, intervals $I_j$, weights $a_N(P)$, and polynomials
$E_\ell$ have the definitions in Section~\ref{sec:GZ}. Define
\begin{align}
 \mathcal P^{(3)}_{r,N_u}(\chi)
 &=\sum_{\deg P\in I_r}
     \frac{a_{N_u}(P)\chi(P)}{\sqrt{|P|}}
   +\mathbf1_{\{r=0\}}\frac12\sum_{\deg P\le B}
                                  \frac{\chi(P)^2}{|P|},
                                  \label{cubic:blocks}\\
 M_{3,G}(\chi)
 &=\prod_{r=0}^J
 E_{\ell_r}\!\left(-\frac{\mathcal P^{(3)}_{r,N_J}(\chi)}2\right)^2
 =\sum_{\deg h\le L}\frac{b_m(h)\chi(h)}{\sqrt{|h|}}.
 \label{cubic:mollifier}
\end{align}
The prime-square term is common to all scales. Since
$\chi^2=\overline\chi$ is primitive and nonprincipal, Weil's theorem
gives
\[
 \left|\sum_{\deg P=d}\chi(P)^2\right|
 \ll_q\frac{mq^{d/2}}d.
\]
The proof of Lemma~\ref{lem:GZshort} therefore yields
\begin{equation}\label{cubic:short-block}
 \log|L_q(1/2,\chi)|
 \le\Rea\sum_{r=0}^j\mathcal P^{(3)}_{r,N_j}(\chi)
       +\theta_j^{-1}+2.
\end{equation}
Indeed, the square-term replacement has error
$O_q(B/N_j+mq^{-B/2}/(B+1))=o(1)$, and all higher prime powers
have a bounded absolute contribution.

The coefficient argument in Lemma~\ref{lem:GZlength} gives
$|b_m(h)|\le1$ and allows the choice $L=2\sum_{r=0}^J\ell_rN_r$.
Since $\ell_r\le100\theta_r^{-3/4}+2$, summing the geometric
series with the parameters \eqref{eq:GZparameters} gives
\begin{equation}\label{cubic:length}
 \frac Lm\le
 \frac{400e^{-51/4}}{1-e^{-1/4}}
 +\frac{8e^{-51}}{1-e^{-1}}
 <0.005248338.
\end{equation}
More precisely, if the multiplicative function $H$ is defined by
\[
 \sum_{v\ge0}H(P^v)x^v=e^{x+x^2/2},
\]
then $|b_m(h)|\le H(h)\le1$. For each fixed $h$, the coefficients
converge to $b_\infty(h)$, where
\begin{equation}\label{cubic:limit-coefficients}
 \sum_{v\ge0}b_\infty(P^v)x^v=e^{-x-\frac{x^2}{2}}.
\end{equation}

Write
\[
 h=CS^2E^3,\qquad h^\ast=C^2S,
\]
where $C,S$ are square-free and coprime; no coprimality with $E$ is
required. Thus $hf$ is a cube precisely when $f=h^\ast A^3$.

\begin{proposition}[David, Florea, and Lal\'in]\label{cubic:twisted}
For $\deg h<(1/10-\varepsilon)g$,
\begin{align}\label{cubic:twisted-formula}
 \sum_{\chi\in\mathcal C_3(g)}\chi(h)L_q(1/2,\chi)
 ={}&q^{2m}\mathcal M_{3,q}(h)
 +O_{q,\varepsilon}\!\left(q^{7g/8+\deg h/4+\varepsilon g}\right),
 \\
 \mathcal M_{3,q}(h)
 ={}&\frac{c_q}{\sqrt{|h^\ast|}}
       \sum_{A\in\Aq}\frac{\rho(hA)}{|A|^{3/2}}.
       \label{cubic:moment-constant}
\end{align}
\end{proposition}
\begin{proof}
This is \cite[Proposition~8.1]{DFL21} in our notation. To identify the
constant, the local factor for $P\nmid h$ in the $A$-sum is
\[
 F_P=1+\frac{\rho_P|P|^{-3/2}}{1-|P|^{-3/2}}.
\]
For $P\mid h$, the correction relative to $F_P$ is
$\rho_P/(1-(1-\rho_P)|P|^{-3/2})$. It equals one at odd-degree
primes and
\[
 \frac1{1+2|P|^{-1}(1-|P|^{-3/2})}
\]
at even-degree primes. This is the factor in
\cite[(8.15)]{DFL21}. Since $\sqrt{|h^\ast|}=|C|\sqrt{|S|}$,
multiplying by $c_q\prod_PF_P$ gives the normalization in
\cite[(8.14)]{DFL21}.
\end{proof}

\subsection{The mollified first moment}

For $\ell=3$, the factor in \eqref{intro:arithmetic} is
\begin{equation}\label{cubic:arithmetic}
 A_{3,q,G}=\prod_P\left(1+\rho_P\sum_{k\ge1}
                              \frac{d_{3k}}{|P|^{3k/2}}\right),
 \qquad 1<A_{3,q,G}\le\zeta_q(3/2).
\end{equation}
Its first nonconstant local term is $\rho_P/(3|P|^{3/2})$.

\begin{proposition}\label{cubic:first}
The mollifier \eqref{cubic:mollifier} satisfies
\begin{equation}\label{cubic:mollified-first}
 \sum_{\chi\in\mathcal C_3(g)}L_q(1/2,\chi)M_{3,G}(\chi)
 =\bigl(c_qA_{3,q,G}+o(1)\bigr)q^{2m}.
\end{equation}
\end{proposition}
\begin{proof}
Put $\lambda=0.005248338$. For sufficiently large $g$,
\eqref{cubic:length} gives $L<\lambda m<g/20$.
Choose a fixed sufficiently small $\varepsilon>0$, with
$\varepsilon<1/20$, so that Proposition~\ref{cubic:twisted}
applies to every twist. Its total error is
\[
 q^{7g/8+\varepsilon g}
 \sum_{\deg h\le L}|b_m(h)||h|^{-1/4}
 \ll_q q^{7g/8+3L/4+\varepsilon g}
 \ll_q q^{(7/4+3\lambda/4+2\varepsilon)m}=o(q^{2m}),
\]
since
\[
 \frac74+\frac34\lambda<1.753937<2.
\]
Thus the published twisted first moment has sufficient uniformity
for the parameters \eqref{eq:GZparameters}.
The normalized main term is
\[
 c_q\sum_{h,A\in\Aq}
 \frac{b_m(h)\rho(hA)}{\sqrt{|hh^\ast|}\,|A|^{3/2}}.
\]
Dominated convergence follows from
\[
 \sum_{h\in\Aq}\frac{H(h)}{\sqrt{|hh^\ast|}}
 =\prod_P\left(1+\sum_{v\ge1}
 H(P^v)|P|^{-3\lceil v/3\rceil/2}\right)<\infty
\]
and the absolute convergence of $\zeta_q(3/2)$.
At each prime, the total exponent in $h$ and the $L$-series must be
divisible by three. By \eqref{cubic:limit-coefficients} and
\eqref{eq:GZlocal}, the local sum is
\[
 1+\rho_P\sum_{k\ge1}
    \left(\sum_{v=0}^{3k}b_\infty(P^v)\right)|P|^{-3k/2}
 =1+\rho_P\sum_{k\ge1}d_{3k}|P|^{-3k/2}.
\]
Multiplication gives \eqref{cubic:mollified-first}.
\end{proof}

\subsection{The exceptional set}

Define
\begin{equation}\label{cubic:good-set}
 \mathcal T^{(3)}=
 \left\{\chi\in\mathcal C_3(g):
 |\mathcal P^{(3)}_{r,N_u}(\chi)|\le K_r
 \text{ for all }0\le r\le u\le J\right\}.
\end{equation}
Let $\mathcal B_n^{(3)}$ consist of characters for which $n$ is the
first block violating this condition. On the good set,
\eqref{cubic:short-block} and \eqref{eq:GZpointTaylor} give
\begin{equation}\label{cubic:pointwise}
 |L_q(1/2,\chi)M_{3,G}(\chi)|
 \le e^{\theta_J^{-1}+3}<e^{e^{52}+3}.
\end{equation}

We check the probability and Taylor estimates with
\eqref{eq:GZparameters}. The probability bounds in
Section~\ref{sec:GZ} use independent, centered variables of modulus
at most one, and therefore hold for \eqref{cubic:model}.
In the first block, $\E X_P=\E X_P^2=0$; grouping the two terms
at a prime gives a variance parameter at most $3\log m$.
For the Taylor estimate, the prime sum has coefficient square sum
at most $1+\log(2m)$, and the square sum has modulus at most
$(1+\log B)/2$. Thus $\ell_0\asymp(\log m)^{3/2}$ satisfies the
hypothesis of Lemma~\ref{lem:GZTaylor} for sufficiently large $m$.
For every later block the variance is at most two, there is no
square term, and
\[
 \ell_r\ge100e^{153/4}e^{3(J-r)/4}>3\cdot10^{10}.
\]
The same lower bound also shows, uniformly in $J$, that
$\sum_{r=0}^Je^{-\ell_r}<10^{-3}$.
Consequently both Taylor estimates and the matched-block bound
\eqref{eq:GZmatched} hold with these cubic parameters.

\begin{proposition}\label{cubic:exceptional}
For every fixed $A>0$,
\begin{align}\label{cubic:exceptional-bound}
 \sum_{\chi\notin\mathcal T^{(3)}}
       |L_q(1/2,\chi)M_{3,G}(\chi)|
 \le{}&q^{2m}\frac{e^{20}}{(1-e^{-1})^2}
                 e^{-\frac{0.03}{\theta_J}}
       +O_{q,A}(q^{2m}m^{-A}).
\end{align}
\end{proposition}
\begin{proof}
For $n\ge1$, put $s_n=\lfloor1/(8\theta_n)\rfloor$. On
$\mathcal B_n^{(3)}$, apply \eqref{cubic:short-block} at scale
$N_{n-1}$ and majorize the earlier exponentials by their Taylor
polynomials. For $u\ge n$, the resulting modulus square is that of
\begin{align*}
 H_{n,u}(\chi)={}&
 \left(\frac{\mathcal P^{(3)}_{n,N_u}(\chi)}{K_n}\right)^{s_n}
 \prod_{r<n}E_{\ell_r}\!\left(
           \frac{\mathcal P^{(3)}_{r,N_{n-1}}(\chi)}2\right)
 \prod_{r=0}^J E_{\ell_r}\!\left(
          -\frac{\mathcal P^{(3)}_{r,N_J}(\chi)}2\right).
\end{align*}
Its degree $D$ satisfies
\begin{equation}\label{cubic:degree-budget}
 D\le s_nN_n+L,\qquad
 2D\le\frac m2+2L<0.510497m<m.
\end{equation}
Thus \eqref{cubic:finite-average} applies after the indicator is
bounded by $\sum_{u=n}^J|\mathcal P^{(3)}_{n,N_u}/K_n|^{2s_n}$.

For the preceding blocks, the common prime-square term cancels,
and the total variance of the difference of prime weights is at
most four. The argument for \eqref{eq:GZmatched} bounds their
product of expectations by $e^3$. In block $n$, the two prime sums
have variance at most two, and the moment and exponential bounds
give
\[
 \E\left\{|Z|^{2s_n}|E_{\ell_n}(-W/2)|^2\right\}
 \le4e^4(16s_n)^{s_n}.
\]
Each subsequent block contributes at most $4e^2<e^4$. Consequently,
\begin{align}\label{cubic:single-block}
 q^{-2m}\sum_{\chi\in\mathcal B_n^{(3)}}
        |L_q(1/2,\chi)M_{3,G}(\chi)|
 \le{}&e^{20}(J-n+1)e^{\theta_{n-1}^{-1}+4(J-n)}
          \left(\frac{16s_n}{K_n^2}\right)^{s_n}.
\end{align}
This is the same bound as \eqref{eq:GZbadblock}.
Writing $t=J-n$, the inequalities
$\theta_n=\theta_Je^{-t}$ and $s_n\ge1/(9\theta_n)$ give
\[
 \theta_{n-1}^{-1}+4t+s_n\log\frac{16s_n}{K_n^2}
 \le\frac{e+(\log2)/9-(51+t)/18}{\theta_n}+4t
 \le-\frac{0.03}{\theta_n}-t.
\]
Indeed,
\[
 \frac{51}{18}-e-\frac{\log2}{9}>0.038035>0.03,
\]
and $\theta_n\le e^{-51}$ ensures that $t/(18\theta_n)\ge5t$.
Since $\theta_n\le\theta_J$, summing
$(t+1)e^{-0.03/\theta_n-t}$ over $t\ge0$ and using
$\sum_{t\ge0}(t+1)e^{-t}=(1-e^{-1})^{-2}$ yields the first
term in \eqref{cubic:exceptional-bound}.

For the initial block, take $s_0=\lceil K_0\rceil$.
Since $K_0=(\log m)^{3/2}$ and $2s_0N_0=o(m)$,
finite orthogonality and the bound $V_0\le3\log m$ imply
\begin{equation}\label{cubic:initial-count}
 q^{-2m}|\mathcal B_0^{(3)}|
 \le4(J+1)\left(\frac{12s_0\log m}{K_0^2}\right)^{s_0}
 \le e^{-c(\log m)^{3/2}\log\log m}
\end{equation}
for some absolute $c>0$.
The polynomial inside the modulus square representing $|M_{3,G}|^4$
has degree at most $2L$. By \eqref{cubic:length},
$4L/m<0.020994<1$, so another application of
\eqref{cubic:finite-average} gives
\begin{equation}\label{cubic:mollifier-fourth}
 q^{-2m}\sum_{\chi\in\mathcal C_3(g)}|M_{3,G}(\chi)|^4
 \le4^{J+1}e^{48\log m+32J}\ll m^{50}.
\end{equation}
H\"older's inequality, \eqref{cubic:fourth}, and
\eqref{cubic:initial-count}--\eqref{cubic:mollifier-fourth} now give
\begin{align*}
 \sum_{\chi\in\mathcal B_0^{(3)}}|L_q(1/2,\chi)M_{3,G}(\chi)|
 &\le |\mathcal B_0^{(3)}|^{1/2}
 \left(\sum_{\chi\in\mathcal C_3(g)}|L_q(1/2,\chi)|^4\right)^{1/4}
 \left(\sum_{\chi\in\mathcal C_3(g)}|M_{3,G}(\chi)|^4\right)^{1/4}\\
 &\ll_{q,A}q^{2m}m^{-A}.
\end{align*}
This proves \eqref{cubic:exceptional-bound}.
\end{proof}

\subsection{An improved lower bound}

\begin{theorem}\label{cubic:positive}
For every fixed odd prime power $q\equiv2\pmod3$,
\begin{equation}\label{cubic:positive-bound}
 \liminf_{\substack{g\to\infty\\2\mid g}}
 \frac{\#\{\chi\in\mathcal C_3(g):L_q(1/2,\chi)\ne0\}}
      {|\mathcal C_3(g)|}
 \ge\frac{A_{3,q,G}}{2 e^{e^{52}+3}}
 >\frac1{2 e^{e^{52}+3}}
 > e^{-e^{53}}.
\end{equation}
\end{theorem}
\begin{proof}
The explicit exceptional-set constant is bounded by
\[
 \frac{e^{20}}{(1-e^{-1})^2}e^{-0.03e^{51}}
 <\frac18\cdot\frac{64}{125}<\frac18c_qA_{3,q,G}.
\]
Subtracting \eqref{cubic:exceptional-bound} from
\eqref{cubic:mollified-first} therefore gives, for sufficiently
large $m$,
\[
 \left|\sum_{\chi\in\mathcal T^{(3)}}
                L_q(1/2,\chi)M_{3,G}(\chi)\right|
 \ge\frac12c_qA_{3,q,G}q^{2m}.
\]
By \eqref{cubic:pointwise}, every nonzero summand has modulus at
most $e^{e^{52}+3}$. The triangle inequality and
\eqref{cubic:count} give the first inequality in
\eqref{cubic:positive-bound}. The remaining inequalities follow
from $A_{3,q,G}>1$ and $e^{52}+3+\log2<e^{53}$.
\end{proof}

The bound improves $0.4718 e^{-e^{182}}$ from
\cite[Remark~1.8]{DFL21} for the same full cubic family. Together
with Theorems~\ref{thm:GZpositive} and~\ref{sextic:positive}, it proves
Theorem~\ref{thm:main-non-vanishing}. The improvement follows from
\eqref{eq:GZparameters} and the exceptional-set estimate, using
only the twisted first moment in Proposition~\ref{cubic:twisted}.
The constants are not optimized; the limits keep $q$ fixed, and no
explicit starting genus is specified.

\end{document}